\documentclass[12pt]{amsart}
\usepackage{amsthm,amsfonts,amssymb,amscd,mathrsfs}
\usepackage{bezier,longtable,amstext}
\usepackage{amssymb}

\usepackage[all]{xy}

\newcommand\codim{\text{codim}}

\newcommand{\Pic}{\operatorname{Pic}\nolimits}

\newcommand\Span{\text{Span}}

\renewcommand{\phi}{\varphi}

\newtheorem{theorem}{Theorem}[section]
\newtheorem{proposition}[theorem]{Proposition}

\newtheorem{lemma}[theorem]{Lemma}
\newtheorem{corollary}[theorem]{Corollary}
\newtheorem{sub}[subsection]{}

\theoremstyle{definition}
\newtheorem{definition}[theorem]{Definition}
\newtheorem{example}[theorem]{Example}
\newtheorem{proposition-definition}[theorem]{Proposition-Definition}

\newtheorem{remark}[theorem]{Remark}

\nonstopmode

\begin{document}

\begin{flushright}
To the memory of Andrey Todorov
\end{flushright}

\title{Linear ind-Grassmannians}

\author[I.~Penkov]{\;Ivan~Penkov}

\vspace{1cm}

\address{
Jacobs University Bremen \\
School of Engineering and Science, Campus Ring 1, 28759 Bremen, Germany}
\email{i.penkov@jacobs-university.de}

\author[A.~Tikhomirov]{\;Alexander S.~Tikhomirov}

\address{
Department of Mathematics\\
State Pedagogical University\\
Respublikanskaya Str. 108
\newline 150 000 Yaroslavl, Russia}
\email{astikhomirov@mail.ru}

\maketitle

\thispagestyle{empty}

\begin{abstract}
We consider ind-varieties obtained as direct limits of chains of embeddings
$X_1\stackrel{\phi_1}{\hookrightarrow}\dots\stackrel{\phi_{m-1}}{\hookrightarrow}
X_m\stackrel{\phi_m}{\hookrightarrow}X_{m+1}\stackrel{\phi_{m+1}}{\hookrightarrow}\dots$,
where each $X_m$ is a Grassmannian or an isotropic Grassmannian (possibly mixing Grassmannians and isotropic
Grassmannians), and the embeddings $\phi_m$ are linear in the sense that they induce
isomorphisms of Picard groups. We prove that any such ind-variety is isomorphic to one of certain standard
ind-Grassmannians and that the latter are pairwise non-isomorphic ind-varieties.

2010 Mathematics Subject Classification, Primary 14A10, 14M15.

Bibliography: 8 items.


\thanks{}
\end{abstract}

\section{Introduction}\label{sec1}
\label{Introduction}
The Barth--Van de Ven--Tyurin--Sato Theorem claims that any finite rank vector bundle on the infinite complex
projective space $\mathbf{P}^\infty$ is isomorphic to a direct sum  of line bundles. For rank two bundles this theorem
has been proved by Barth and Van de Ven in \cite{BV}, and in the general case the theorem has been proved by Tyurin in
\cite{T} and Sato in [S1]. In the last decade we have studied more general ind-varieties for which the result
holds true \cite{PT1}, \cite{PT2} \cite{DP}.

This study has naturally led us to the problem of constructing non-isomorphic ind-varieties arising as direct limits
of given classes of embeddings of projective varieties. In the present
note we address a classification problem along those lines:
we consider linear embeddings of Grassmannians, i.e.
embeddings $i:X_1\hookrightarrow X_2$ of a Grassmannian $X_1$ into a Grassmannian $X_2$ satisfying the condition
$i^*\mathcal{O}_{X_2}(1)\simeq\mathcal{O}_{X_1}(1)$, and determine
how many non-isomorphic ind-varieties can be obtained from such embeddings. Moreover,
we consider also orthogonal and symplectic Grassmannians (i.e. isotropic Grassmannians arising from non-degenerate
orthogonal or
symplectic forms) and define a \textit{linear ind-Grassmannian} as an ind-variety arising as the direct limit
$\underset{\longrightarrow}\lim\ X_n$ of any chain of linear embeddings
$$
X_1\hookrightarrow X_2\hookrightarrow\dots\hookrightarrow X_m \hookrightarrow X_{m+1}\hookrightarrow\dots
$$
of Grassmannians, some or all of them orthogonal or symplectic.

Our main result (Theorem 2, see Section \ref{sec5}) states that each linear ind-Grassmannian is isomorphic (as an
ind-variety) to one of the standard ind-Grassmannians introduced in \cite{DiP}. In particular, any linear
ind-Grassmannian is a homogeneous space
of one of the three classical ind-groups $SL(\infty)$, $O(\infty)$, $Sp(\infty)$. We also prove in Theorem 2 that the
standard ind-Grassmannians are pairwise non-isomorphic. To make the note self-contained, we do not rely on the
article \cite{DiP}, but introduce (in Section \ref{stand Gr} below) the standard ind-Grassmannians in terms of
explicit chains of embeddings.

The main tool we use in Theorem 2 is Theorem 1 (see Section \ref{Lin emb}) which describes linear morphisms of
Grassmannians, as well as isotropic Grassmannians.

In the related paper \cite{PT3} we return to the original question of extending the generality of the Barth-Van de Ven-
Tyurin-Sato theorem. There we give the list of linear ind-Grassmannians on which a bundle of finite rank is isomorphic
to a direct sum of line bundles.

\textbf{Acknowledgement.} I.P. acknowledges the hospitality of the Max-Planck-Institute for Mathematics in Bonn, as well
as partial DFG support through "DFG Schwerpunkt 1388, Darstellungstheorie". A.S.T. acknowledges partial support from
Jacobs University Bremen.

\vspace{1cm}

\section{Preliminaries}\label{sec2}
\vspace{0.5cm}

\begin{sub}{\bf Notation and conventions.}\label{subsection 2.1}
\rm
Recall that $\mathbb{N}=\{0,1,2,...\}$. We set $\mathbb{Z}_+=\{1,2,3,...\}$. All vector spaces and algebraic varieties
are defined over an algebraically closed field $\mathbb{F}$ of characteristic 0.
The superscript ${}^*$ indicates dual space or dual vector bundle as well as inverse image. If $X$ is a projective
variety with Picard group isomorphic to $\mathbb{Z}$, then $\mathcal{O}_X(1)$ stands for the ample generator of the
Picard group.

By $G(k,V),\ 1\le k\le\dim V,$ we denote the Grassmannian of $k$-dimensional subspaces of a finite-dimensional vector
space $V$. For $k=1$, $G(k,V)=\mathbb{P}(V)$. Furthermore, $\mathcal{O}_{G(k,V)}(1)\cong\wedge^kS_{G(k,V)}^*$, where
$S_{G(k,V)}$ is the tautological bundle on $G(k,V)$,  and $\Pic G(k,V)\cong\mathbb{Z}\mathcal{O}_{G(k,V)}(1)$.

In what follows we will consider, both symmetric and symplectic, quadratic forms $\Phi$ on $V$.
Under the assumption that $\Phi$ is fixed, we set $W^\bot:=\{v\in V|\ \Phi(v,w)=0$ for any
$w\in W\}$ for any subspace $W\subset V$. Recall
that $W$ is {\it isotropic} (or $\Phi$-\textit{isotropic}) if $W\subset W^\bot$.
\end{sub}

\begin{sub}{\bf Linear morphisms.}\label{subsection 2.2}
\rm

\begin{definition}\label{lin mor}
We call a morphism $\phi: X\to Y$ of algebraic varieties (or ind-varieties) \textit{linear} if $\phi$ induces an
epimorphism of Picard groups $\phi^*:\Pic Y\to\Pic X$.
\end{definition}

In this paper we focus on linear embeddings $\phi: X\to Y$ of Grassmannians or isotropic Grassmannians. In this case
$\phi$ is linear iff $\phi^*\mathcal{O}_Y(1)\cong\mathcal{O}_X(1)$.
By a \textit{projective space on}, or \textit{in}, a variety (or ind-variety) $X$ we
understand a linearly embedded subvariety $Y$ of $X$ isomorphic to a projective space.
Note that the Pl\"ucker embedding $G(k,V)\hookrightarrow\mathbb{P}(H^0(\mathcal{O}_{G(k,V)}(1))^*)$ is a linear
morphism.

By a \textit{quadric on} $X$ of dimension $m\ge3$ we understand a linearly embedded subvariety
$Y$ of $X$ isomorphic to a smooth $m$-dimensional quadric. By a quadric on $X$ of dimension 2
we understand the image of an embedding
$i:\mathbb{P}^1\times\mathbb{P}^1\hookrightarrow X$ such that $i^*\mathcal{O}_X(1)\simeq
\mathcal{O}_{\mathbb{P}^1}(1)\boxtimes\mathcal{O}_{\mathbb{P}^1}(1)$.
By a quadric on $X$ of dimension 1, or a \textit{conic} on $X$,
we understand the image of an embedding $i:\mathbb{P}^1\hookrightarrow X$ such that
$i^*\mathcal{O}_X(1)\simeq\mathcal{O}_{\mathbb{P}^1}(2)$. Given a quadric $Q$, we set $\mathbb{P}_Q=\mathbb{P}
(H^0(\mathcal{O}_Q(1))^*)$ for $m\ge3$, and respectively
$\mathbb{P}_Q=\mathbb{P}(H^0(\mathcal{O}_{\mathbb{P}^1}(1)\boxtimes\mathcal{O}_{\mathbb{P}^1}(1))^*)$,
$\mathbb{P}_Q=\mathbb{P}(H^0(\mathcal{O}_{Q}(2))^*)$ for $m=2,1$. Then $Q$ is canonically embedded into $\mathbb{P}_Q$.
\end{sub}

\begin{sub}{\bf Orthogonal Grassmannians.}\label{linear ortho ind-Grassm}
\rm

Let $\Phi\in S^2V^*$ be a non-degenerate symmetric form on $V$.
For $\dim V\ge3$ and  $1\le k\le[\frac{\dim V}{2}]$, the \textit{orthogonal Grassmannian} $GO(k,V)$ is defined as the
subvariety of $G(k,V)$ consisting of $\Phi$-isotropic $k$-dimensional subspaces of $V$.
Unless $\dim V=2n$, $k=n$, $GO(k,V)$ is a smooth irreducible variety.
For $\dim V=2n$, $k=n$, $GO(k,V)$ is smooth and has two irreducible components, both of which are isomorphic to
$GO(n-1,V')$ where $\dim V'=2n-1$.

The orthogonal Grassmannian $GO(k,V)$ has the following dimension:
$$
\dim GO(k,V)=\Bigl\{
\begin{array}{lll}
2kn-\frac{1}{2}k(3k+1) & \textrm{for}\ \ 1\le k\le n, & \dim V=2n, \\
k(2n+1)-\frac{1}{2}k(3k+1) & \textrm{for}\ \ 1\le k\le n, & \dim V=2n+1.
\end{array}
$$
Moreover, for any $V$ and $1\le k\le[\frac{\dim V}{2}]$, $k\ne\frac{\dim V}{2}-1$,
$$
\Pic GO(k,V)=\mathbb{Z}\mathcal{O}_{GO(k,V)}(1),
$$
where the sheaf $\mathcal{O}_{GO(k,V)}(1)$ posesses the following property: if $t: GO(k,V)\hookrightarrow G(k,V)$
is the tautological embedding, then
$$
t^*\mathcal{O}_{G(k,V)}(1)\cong\Bigl\{
\begin{array}{lll}
\mathcal{O}_{GO(k,V)}(1) & for & 1\le k\le [\frac{\dim V}{2}]-1, \\
\mathcal{O}_{GO(k,V)}(2) &  for & k=[\frac{\dim V}{2}].
\end{array}
$$
In what follows we will think of $GO(n-1,V)$ for $\dim V=2n$ as a variety of isotropic flags rather than as an
orthogonal Grassmannian. In addition, we exclude the case $\dim V=2n,\ k=n$ from consideration.
More precisely, when writing $GO(k,V)$ below
we assume that $\dim V\ge7$ and $k\ne\frac{\dim V}{2}$, $k\ne\frac{\dim V}{2}-1$.

For $k<n=[\frac{\dim V}{2}]$ on $GO(k,V)$ there is a single family of maximal projective spaces of dimension $k$
with base $PO_{\alpha}(k,V)$. There is also a family of $(\dim V-2k)$-dimensional maximal quadrics not contained
in projective spaces on $GO(k,V)$. We denote the base of this family by $QO_{\beta}(k,V)$.
In addition, for $k\le[\frac{\dim V}{2}]-2$ there is a family of 4-dimensional maximal quadrics not contained
in projective spaces on $GO(k,V)$. We denote the base of this family by $QO_{\gamma}(k,V)$.

For $k=n$ on $GO(k,V)$ there is a single family of maximal projective spaces of
dimension $[\frac{\dim V-1}{2}]$ with irreducible base $PO_{\alpha}(k,V)$. Furthermore, if $\dim V=2n+1$ and
$1\le k\le n$, on $GO(k,V)$ there is a single family of maximal projective spaces of dimension $n-k$ with irreducible
base $PO_{\beta}(k,V)$.
The varieties $PO_{\alpha}(k,V)$, $PO_{\beta}(k,V)$, $QO_{\beta}(k,V)$ and $QO_{\gamma}(k,V)$ are described by
the following lemma.

\begin{lemma}\label{Bkn}
(i) If $1\le k\le n-1$, then each $k$-dimensional projective space on $GO(k,V)$ is of the form
\begin{equation}\label{PkI}
\{V_k\in GO(k,V)|\ V_k\subset V_{k+1}\}\simeq\mathbb{P}(V_{k+1}^*)
\end{equation}
for a fixed $(k+1)$-dimensional isotropic subspace $V_{k+1}$. Consequently, for $k\ne\frac{\dim V}{2}-2$,
$PO_{\alpha}(k,V)$ is isomorphic to $GO(k+1,V)$.

(ii) If $1\le k\le n-1$, then $PO_{\beta}(k,V)$ is isomorphic to the variety of isotropic $(k-1,n)$-flags in $V$,
and for any point $(V_{k-1}\subset V_n)\in PO_{\beta}(k,V)$ the corresponding projective space on $GO(k,V)$ is
\begin{equation}\label{Pn-k}
\{V_k\in GO(k,V)|\ V_{k-1}\subset V_k\subset V_n\}\simeq\mathbb{P}(V_n/V_{k-1}).
\end{equation}

(iii) If $k=n$, then $PO_{\alpha}(n,V)$ is isomorphic to $GO(n,V)$,
and for any point $V_n\in GO(n,V)$ the corresponding projective space on $GO(n,V)$ is
\begin{equation}\label{PnI}
\{V'_n\in GO(n,V)|\ \dim(V'_n\cap V_n)=n-1\}\simeq\mathbb{P}(V/V_n).
\end{equation}

(iv) If $1\le k\le n$, then $QO_{\beta}(k,V)$ is isomorphic to $GO(k-1,V)$, and for any point
$V_{k-1}\in GO(k-1,V)$ the corresponding quadric on $GO(k,V)$ is
\begin{equation}\label{Q}
\{V_k\in GO(k,V)|V_k\supset V_{k-1}\}\simeq GO(1,V_{k-1}^\bot/V_{k-1}).
\end{equation}

(v) $QO_{\gamma}(k,V)$ is isomorphic to the variety of isotropic $(k-2,k+2)$-flags in $V$, and for any point
$(V_{k-2}\subset V_{k+2})\in QO_{\gamma}(k,V)$ the corresponding quadric on $GO(k,V)$ is
\begin{equation}\label{Q1}
\{V_k\in GO(k,V)|V_{k-2}\subset V_k\subset V_{k+2}\}.
\end{equation}

(vi) Any maximal quadric on $GO(k,V)$ is either of the form (\ref{Q}) or (\ref{Q1}), or lies in a projective space on
$GO(k,V)$.
\end{lemma}

\begin{proof}
We leave the proof of (i)-(v) to the reader and give an outline of the proof of (vi).
Let $Q$ be a quadric on $GO(k,V)$ and let $G$ be the variety of projective planes in $\mathbb{P}_Q$.
In $G$ there is a dense open subset $U=\{\mathbb{P}^2\in G\ |\ \mathbb{P}^2\cap Q$ is a conic$\}$, and if
$\mathbb{P}^2\cap Q=C$ then $\mathbb{P}_C=\mathbb{P}^2\in U$.
In what follows, by a slight abuse of notation, we will indicate this latter fact by writing $C\in U$.

Let $F$ be the variety of $(1,k)$-isotropic flags in $V$
with projections $\mathbb{P}(V)\overset{pr_1}\leftarrow F\overset{pr_2}\to GO(k,V)$. For any $C\in U$ set
$\tilde{K}_C:=pr_2^{-1}(C)$, $K_C:=pr_1(\tilde{K}_C)$ and let $p_C:=pr_1|_{\tilde{K}_C}:\ \tilde{K}_C\to K_C$ be the
projection. There are three possibilities:

(a) there exists a dense open subset $U'$ in $U$ such that, for any $C\in U'$, $p_C$ is an isomorphism and $K_C$ is a
quadratic cone with vertex $S=\mathbb{P}(V_{k-1}(C))$ for some subspace $V_{k-1}(C)$ in $V$,

(b) there exists a dense open subset $U'$ in $U$ such that, for any $C\in U'$, $p_C$ is an isomorphism and $K_C$ is a
quadratic cone with vertex $S=\mathbb{P}(V_{k-2}(C))$ for some subspace $V_{k-2}(C)$ in $V$,

(c) for any $C\in U$, $p_C$ is a double covering and $K_C=\mathbb{P}(V_{k+1}(C))$ for some subspace $V_{k+1}(C)$ of $V$.

Using the fact that $U$ and $U'$ are dense subsets in $G$, one easily checks the following facts.
In case (a) the space $V_{k-1}=V_{k-1}(C)$ does not depend on the conic $C\in U'$ and
$Q\subset GO(1,V_{k-1}^\bot/V_{k-1})\in QO_{\beta}(k,V)$. In case (b)  the space $V_{k-2}=V_{k-2}(C)$ does not depend on
the conic $C\in U'$ and $Q$ is contained in a quadric $\bar{Q}$ given by formula (\ref{Q1}), i.e
$\bar{Q}\in QO_{\gamma}(k,V)$. In case (c) the space $V_{k+1}=V_{k+1}(C)$ does not depend on the conic $C\in U$, so that
$Q\subset\mathbb{P}(V_{k+1}^*)\subset GO(k,V)$.
\end{proof}

In what follows we will sometimes write $\mathbb{P}^k_{\alpha}$ for a maximal projective space on $GO(k,V)$ of the form
(\ref{PkI}) or (\ref{PnI}), and $\mathbb{P}^{n-k}_{\beta}$ for a maximal projective space on $GO(k,V)$ of the
form (\ref{Pn-k}). We will also write $Q^{\dim V-2k}_{\beta}$ for a maximal quadric on $GO(k,V)$ of the
form (\ref{Q}), and $Q^4_{\gamma}$ for a maximal quadric of the form (\ref{Q1}).

\begin{lemma}\label{two possib} Let $1\le k\le n$.

(i) The intersection of any two distinct projective spaces $\mathbb{P}^k_{\alpha}$ and
$(\mathbb{P}^k_{\alpha})'$ (respectively, $\mathbb{P}^{n-k}_{\beta}$ and $(\mathbb{P}^{n-k}_{\beta})'$) on $GO(k,V)$
is either empty or equals a point.

(ii) The intersection of any projective space $\mathbb{P}^k_{\alpha}$ and any quadric
$Q^{\dim V-2k}_{\beta}$ on $GO(k,V)$ is empty, equals a point, or equals a projective line. The intersection of any two
distinct quadrics $Q^{\dim V-2k}_{\beta}$ and $(Q^{\dim V-2k}_{\beta})'$ on $GO(k,V)$ is either empty or equals a point.

(iii) Assume $k\le n-1$. Then the intersection of any two distinct projective spaces $\mathbb{P}^k_{\alpha}$ and $\mathbb{P}^{n-k}_{\beta}$
on $GO(k,V)$ is empty, equals a point, or equals a projective line.

(iv) Assume $k\le n-1$. Then $\mathbb{P}^k_{\alpha}\cap \mathbb{P}^{n-k}_{\beta}=\{V_k\}$ if and only if\
$\mathbb{P}^k_{\alpha}=\mathbb{P}(V_{k+1}^*)$, $\mathbb{P}^{n-k}_{\beta}=\mathbb{P}(V_n/V_{k-1})$ for a configuration
of isotropic vector subspaces $V_{k+1},V_{k-1},V_n$ of $V$ satisfying
$$
V_{k-1}\subset V_k\subset V_n,\ \ \ V_{k+1}\cap V_n=V_k.
$$
\end{lemma}

\begin{proof} Exercise.
\end{proof}

\begin{lemma}\label{family CO}
(i) Let $\mathbb{P}^1$ be a projective line on $GO(k,V)$, $x\not\in\mathbb{P}^1$ be a fixed
point in $GO(k,V)$, and $C\subset GO(k-1,V)$ be an irreducible curve such that, for any $V_{k-1}\in C$, the quadric
$GO(1,V_{k-1}^\bot/V_{k-1})$ on $GO(k,V)$ contains $x$ and intersects $\mathbb{P}^1$. Then $C$
is a projective line on $GO(k-1,V)$.

(ii) Assume $1\le k\le n-1$. Let $\mathbb{P}^1$ be a projective
line on $GO(k,V)$, $x\not\in\mathbb{P}^1$ be a fixed point in $GO(k,V)$, and $C\subset GO(k+1,V)$ be an
irreducible curve such that, for any $V_{k+1}\in C$, the projective space $\mathbb{P}(V_{k+1}^*)$ on $GO(k,V)$ contains
$x$ and intersects $\mathbb{P}^1$. Then $C$ is a projective line on $GO(k+1,V)$.
\end{lemma}

\begin{proof}
(i) Assume $k<n$ and let
\begin{equation}\label{pensil}
\mathbb{P}^1=\{V_k\in V|U_{k-1}\subset V_k\subset U_{k+1}\}
\end{equation}
for a fixed isotropic flag $U_{k-1}\subset U_{k+1}$ in $V$.
Next, let $x=W_k$. Since for any $V_{k-1}\in C$, the quadric $GO(1,V_{k-1}^\bot/V_{k-1})$ contains
the point $x$, we have $V_{k-1}\subset W_k$, and consequently
$$
\Span(\underset{V_{k-1}\in C}{\cup}V_{k-1})=W_k.
$$
The condition that the quadric $GO(1,V_{k-1}^\bot/V_{k-1})$ intersects $\mathbb{P}^1$ shows that
\begin{equation}\label{incl 1}
V_{k-1}\subset V_k\subset U_{k+1},\ \ \ U_{k-1}\subset V_k
\end{equation}
for some $V_k\in\mathbb{P}^1$. In particular,
$$
W_k\subset U_{k+1}.
$$
Note that $U_{k-1}\not\subset W_k$ as otherwise $x\in\mathbb{P}^1$.
Therefore $W_{k-2}:=W_k\cap U_{k-1}$ is a
$(k-2)$-dimensional subspace of $W_k$. Now (\ref{incl 1}) implies that
$C=\{V_{k-1}\in GO(k-1,V)|\ W_{k-2}\subset V_{k-1}\subset W_k\}$, i.e. $C$ is a projective line on $GO(k-1,V)$.

We leave the case $k=n$ to the reader.

(ii) Formula (\ref{pensil}) holds also in this case. Furthermore,
$$
\underset{V_{k+1}\in C}{\cap}V_{k+1}=W_k=x.
$$
For any $V_{k+1}\in C$, the condition that  $\mathbb{P}(V_{k+1}^*)$ intersects $\mathbb{P}^1$ yields
$V_k$ such that
$$
U_{k-1}\subset V_k\subset V_{k+1}.
$$
Therefore $U_{k-1}\subset W_k$. Now if $U_{k+1}\in C$, then for any $V_{k+1}\in C$, $\mathbb{P}(V_{k+1}^*)$ intersects
$\mathbb{P}^1$ in $x$, contrary to the assumption that $x\not\in\mathbb{P}^1$. Hence, $U_{k+1}\not\in C$ and one checks
that $C=\{V_{k+1}\subset V|W_k\subset V_{k+1}\subset W_{k+2}\}$, where
$W_{k+2}:=\Span(W_k,U_{k+1})$ is a $(k+2)$-dimensional subspace of $V$. This means
that $C$ is a projective line on $GO(k+1,V)$.

\end{proof}

\end{sub}

\begin{sub}{\bf Symplectic Grassmannians.}\label{symp ind-Grassm}

\rm
Let now $\Phi\in \wedge^2V^*$ be a non-degenerate symplectic form on $V$, $\dim V=2n$.

Assume $1\le k\le n$. Recall that the \textit{$k$-th symplectic Grassmannian}  $GS(k,V)$ is the smooth irreducible
subvariety of $G(k,V)$ consisting of $\Phi$-isotropic $k$-dimensional subspaces of $V$.
It is well known that
\begin{equation}\label{dim IIkn}
\dim GS(k,V)=2kn-\frac{1}{2}k(3k-1).
\end{equation}
It is also known that $\Pic GS(k,V)=\mathbb{Z}\mathcal{O}_{GS(k,V)}(1)$ and
$\mathcal{O}_{GS(k,V)}(1)=i^*\mathcal{O}_{G(k,V)}(1)$, where $i:GS(k,V)\hookrightarrow G(k,V)$
is the tautological embedding.

\vspace{2mm}

One can see that, for $1\le k\le n-1$, there are two families of maximal
projective spaces on $GS(k,V)$ of respective dimensions $k$ and $2n-2k+1$, with bases
$PS_{\alpha}(k,V)$ and $PS_{\beta}(k,V)$.
For $k=n$ there is a single family $PS_{\beta}(n,V)$ of maximal projective lines on $GS(k,V)$.
\begin{lemma}\label{Bkn sympl}

(i) Let $1\le k\le n-1$. Then $PS_{\alpha}(k,V)$ is isomorphic to $GS(k+1,V)$, and for any point $V_{k+1}\in GS(k+1,V)$
the corresponding projective space on $GS(k,V)$ is
\begin{equation}\label{PkS}
\{V_k\in GS(k,V)|\ V_k\subset V_{k+1}\}\simeq\mathbb{P}(V_{k+1}^*).
\end{equation}

(ii) Let $1\le k\le n$. Then $PS_{\beta}(k,V)$ is isomorphic to $GS(k-1,V)$, and for any point
$V_{k-1}\in GS(k-1,V)$ the corresponding projective space on $GS(k,V)$ is
\begin{equation}\label{Z(V_k)}
\{V_k\in GS(k,V)\ |\ V_{k-1}\subset V_k\subset V_{k-1}^\bot\}\simeq\mathbb{P}(V_{k-1}^\bot/V_{k-1}).
\end{equation}

(iii) If $k=n$, then any maximal projective space on $GS(n,V)$ is a projective line.
\end{lemma}
\begin{proof} Exercise.
\end{proof}

In what follows we will sometimes write $\mathbb{P}^k_{\alpha}$ for a maximal projective space on $GS(k,V)$ of the form
(\ref{PkS}), and $\mathbb{P}^{2n-2k+1}_{\beta}$ for a maximal projective space on $GS(k,V)$ of the
form (\ref{Z(V_k)}) (despite the fact that we use the same notation as in the orthogonal case,
we will carefully distinguish between the two cases).

\begin{lemma}\label{intersect dim=1}
Let $\dim V=2n$, $n\ge2$, and $1\le k\le n-1$.

(i) The intersection of any two distinct projective spaces $\mathbb{P}^k_{\alpha}$ and
$(\mathbb{P}^k_{\alpha})'$ (respectively, $\mathbb{P}^{2n-2k+1}_{\beta}$ and $(\mathbb{P}^{2n-2k+1}_{\beta})'$) on
$GS(k,V)$ is either empty or equals a point.

(ii) The intersection of any two distinct projective spaces $\mathbb{P}^k_{\alpha}$ and
$\mathbb{P}^{2n-2k+1}_{\beta}$ on $GS(k,V)$ is either empty or equals a projective line.

\vspace{3mm}
(iii) The spaces $\mathbb{P}^k_{\alpha}$ and $\mathbb{P}^{2n-2k+1}_{\beta}$ intersect in a projective line
if and only if
$\mathbb{P}^k_{\alpha}=\mathbb{P}(V_{k+1}^*)$, $\mathbb{P}^{2n-2k+1}_{\beta}=\mathbb{P}(V_{k-1}^\bot/V_{k-1})$ for a
flag $V_{k-1}\subset V_{k+1}$ of isotropic subspaces of $V$. Then
$\mathbb{P}^k_{\alpha}\cap\mathbb{P}^{2n-2k+1}_{\beta}=\mathbb{P}(V_{k+1}/V_{k-1})$.
\end{lemma}
\begin{proof} Exercise.
\end{proof}

\begin{lemma}\label{family CS}
(i) Assume $2\le k\le n$.
Let $\mathbb{P}^1$ be a projective
line on $GS(k,V)$, $x\not\in\mathbb{P}^1$ be a fixed point in $GS(k,V)$, and $C\subset GS(k-1,V)$ be
an irreducible curve such that, for any $V_{k-1}\in C$, the projective space $\mathbb{P}(V_{k-1}^\bot/V_{k-1})$ on
$GS(k,V)$ contains $x$ and intersects $\mathbb{P}^1$. Then $C$ is a projective line on $GS(k-1,V)$.

(ii) Assume $1\le k\le n-1$. Let $\mathbb{P}^1$ be a projective
line on $GS(k,V)$, $x\not\in\mathbb{P}^1$ be a fixed point in $GS(k,V)$, and $C\subset GS(k+1,V)$ be an
irreducible curve such that, for any $V_{k+1}\in C$, the projective space $\mathbb{P}(V_{k+1}^*)$ on $GS(k,V)$ contains
$x$ and intersects $\mathbb{P}^1$. Then $C$ is a projective line on $GS(k+1,V)$.

\end{lemma}
\begin{proof}
Very similar to the proof of Lemma \ref{family CO}.
\end{proof}

\vspace{3mm}

\end{sub}

\vspace{1cm}
\section{Linear embeddings of Grassmannians}\label{Lin emb}

\vspace{1cm}

In this section we study linear embeddings of Grassmannians and isotropic Grassmannians.

We start with the following general lemma whose proof we leave to the reader.

\begin{lemma}\label{finite}
Any non-constant morphism of Grassmannians (respectively, orthogonal or symplectic Grassmannians) is finite.
\end{lemma}

\begin{definition}\label{standard gen}
Let $X,X'$ be Grassmannians. An embedding $\phi':X\hookrightarrow X'$
is a \textit{standard extension}, if there are isomorphisms $i_X,i_{X'}$ and an
embedding $\phi:G(k,V)\hookrightarrow G(k',V')$ for $\dim V'\ge\dim V,\ k'\ge k,$ such
that the diagram
\begin{equation}\label{Def 3.1}
\xymatrix{
X \ar@{^{(}->}[r]^-{\phi'}\ar[d]^-{i_X} & X'\ar[d]^-{i_{X'}}\\
G(k,V) \ar@{^{(}->}[r]^-{\phi} & G(k',V') }
\end{equation}
is commutative and $\phi$ is given by the formula
\begin{equation}\label{st ext}
 \phi:V_k\mapsto V_k\oplus W
\end{equation}
for some fixed isomorphism $V'\simeq V\oplus\hat{W}$ and a fixed subspace $W\subset\hat{W}$ of dimension $k'-k$.
\end{definition}

It is easy to see that a standard extension is a linear embedding. Furthermore, if $\mathbb{P}^q$ is a projective space
on $G(k,V)$, then the inclusion $\mathbb{P}^q\hookrightarrow G(k,V)$ is a standard extension.

\begin{example}
Let $V'=V\oplus\hat{W},\ X=G(n-k,V^*),\ X'=G(k',V'),\ W\subset\hat{W}$ be a fixed subspace of dimension $k'-k$,
$\varepsilon: V\to V$ be any autorphism. Then the morphism
$$
X=G(n-k,V^*)\simeq G(k,V)\to G(k',V')=X',
$$
$$
V_{n-k}\mapsto V_{n-k}^*\mapsto\varepsilon(V_{n-k}^*)\oplus W
$$
is a standard extension. Here $i_X$ is the isomorphism $G(n-k,V^*)\simeq G(k,V)$ and $i_{X'}$ is the automorphism
of $X'$ induced by the automorphism $\varepsilon^{-1}\oplus{\rm id}_{\hat{W}}$ of $V'$.
\end{example}

\begin{remark}\label{intrinsic def}
Note that, for a standard extension $\phi':X\to X'$ the dimensions of $V$ and $V'$ are fixed by the respective
isomorphism classes of $X$ and $X'$, however, the choice between $k$ and $\dim V-k$, respectively, $k'$ and
$\dim V'-k'$, in diagram (\ref{Def 3.1}) is made by the morphism $\phi'$.
Furthermore, if $V,V',k,k'$ are chosen, fixing a standard extension $\phi:G(k,V)\to G(k',V')$
for which $i_{G(k,V)}$ and $i_{G(k',V')}$ are automorphisms is equivalent to fixing some linear algebraic data.
More precisely, given such a standard extension $\phi:G(k,V)\hookrightarrow G(k',V'),$ we can recover $W$ by the formula
$W=\underset{V_k\in G(k,V)}\cap \phi(V_k)$. Set $U:=\Span(\underset{V_k\in G(k,V)}\cup \phi(V_k))$.
Then $W\subset U$ is a flag in $V'$ and $\phi$ determines a surjective linear operator
$\underline{\phi}:U\to V$ with kernel $W$, such that $(\underline{i})^{-1}(V_k)=\phi(V_k)$
for any $k$-dimensional subspace $V_k\in G(k,V)$. It is easy to check that fixing the standard extension $\phi$ is
equivalent to fixing the triple $(W,U,\underline{\phi})$.

In what follows we will write somewhat informally $\phi:G(k,V)\hookrightarrow G(k',V')$ for a general standard
extension, while we will speak about a \textit{strict} standard extension when $i_{G(k,V)}$ and $i_{G(k',V')}$ are
automorphisms. Given a strict standard extension $\phi:G(k,V)\hookrightarrow G(k',V')$, the isomorphism
$V'\simeq V\oplus\hat{W}$ can always be changed so that $\phi$ is given simply by formula (\ref{st ext}).

\end{remark}

We now give a similar definition of a standard extension of isotropic Grassmannians (cf. [DP] and [PT1, section 3]).
\begin{definition}\label{lin mor isotr}
An embedding $\phi:GO(k,V)\hookrightarrow GO(k',V')$ is a \textit{standard extension} if
$\phi$ is given by formula (\ref{st ext}) for some orthogonal isomorphism
$V'\simeq V\oplus\hat{W}$ and a fixed isotropic subspace $W$ of $\hat{W}$.
A standard extension of symplectic Grassmannians is defined in the same way by replacing $GO$ with $GS$,
and the orthogonal isomorphism $V'\simeq V\oplus\hat{W}$ by a symplectic isomorphism $V'\simeq V\oplus\hat{W}$.
\end{definition}

Under an orthogonal isomorphism (respectively, symplectic isomorphism) we mean an isomorphism of vector spaces together
with an isomorphism of forms $\Phi'\simeq\Phi\oplus\hat{\Phi}$, where $\Phi$ is a fixed symmetric (respectively,
symplectic) form on $V$, $\Phi'$ is a fixed (respectively, symplectic) form on $V'$, and $\hat{\Phi}$ is a fixed
symmetric (respectively, symplectic) form on $\hat{W}$.

\begin{remark}\label{intrinsic def2}
A standard extension of isotropic Grassmannians can be defined as follows: consider a flag of subspaces
$W\subset U$ of $V'$, where $W$ is isotropic and there is a surjective linear operator $\underline{\phi}:U\to V$ with
kernel $W$, such that the form $\underline{\phi}^*\Phi$ coincides with the form induced on $U$ by the form $\Phi'$.
This datum defines an embedding $GO(k,V)\to GO(k',V')$ (respectively, $GS(k,V)\to GS(k',V')$) by the formula
$$
\phi:\ V_k\mapsto(\underline{\phi})^{-1}(V_k)\subset U\subset V'\ \ \ \textrm{for} \ \ \ V_k\in GO(k,V) \ \ \
(\textrm{resp.,}\ V_k\in GS(k,V)).
$$
Furthermore,
\begin{equation}\label{gen W,U}
W=\cap \phi(V_k),\ \ \ U=\Span(\cup \phi(V_k)),
\end{equation}
where $V_k$ runs over $GO(k,V)$ (respectively, $GS(k,V)$) and the intersection and the union are taken in $V'$.
\end{remark}

\begin{remark}\label{prop of st ext1}
Let $\phi:G(k,V)\to G(k',V')$ be a strict standard extension (respectively, $\phi:GO(k,V)\to GO(k',V')$ or
$\phi:GS(k,V)\to GS(k',V')$ be a standard extension). Then
\begin{equation}\label{stand ineq1}
k'\ge k\ \ \  \textrm{and}\ \ \ \ \dim V'-k'\ge\dim V-k\ge0
\end{equation}
(respectively,
\begin{equation}\label{stand ineq2}
k'\ge k\ \ \  \textrm{and}\ \ \ \frac{1}{2}\dim V'-k'\ge\frac{1}{2}\dim V-k\ge0).
\end{equation}
Indeed, Definition \ref{standard gen} implies $k'-k=\dim W\ge0$.
Next, from $\dim W\le\dim\hat{W}=\dim V'-\dim V$
it follows that $\dim V'-k'=\dim V-k+(\dim\hat{W}-\dim W)\ge\dim V-k$. This proves (\ref{stand ineq1}).
As for (\ref{stand ineq2}), from Definition
\ref{lin mor isotr} we have $k'-k=\dim W\ge0$. Furthermore, as $V_k$ is $\Phi$-isotropic,
$V_{k'}:=V_k\oplus W$ is $\Phi'$-isotropic and $W$ is $\hat{\Phi}$-isotropic, we have
$k\le\frac{1}{2}\dim V,\ k'\le\frac{1}{2}\dim V',\ 0\le\dim W\le\frac{1}{2}\dim\hat{W}=\frac{1}{2}(\dim V'-\dim V)$.
This implies $\frac{1}{2}\dim V'-k'=\frac{1}{2}\dim V-k+\frac{1}{2}\dim\hat{W}-\dim W\ge\frac{1}{2}\dim V-k\ge0$.
\end{remark}

\begin{definition}\label{isotr st ext}
(a) Let $V''$ be an isotropic subspace of $V$. For $\mathbb{Z}_+\ni k\le\dim V''$, we call
the natural inclusions $G(k,V'')\hookrightarrow GO(k,V)$ and $G(\dim V''-k,{V''}^*)\hookrightarrow GO(k,V)$
(respectively, $G(k,V'')\hookrightarrow GS(k,V)$ and $G(\dim V''-k,{V''}^*)\hookrightarrow GS(k,V)$)
\textit{isotropic extensions}.

(b) A \textit{combination of isotropic and standard extensions} is an embedding of the form
$$
GO(k,V)\overset{t}\hookrightarrow G(k,V)\overset{\phi'}\hookrightarrow G(l,U)\overset{\tau}\hookrightarrow
GO(l,\tilde{U})\overset{\phi''}\hookrightarrow GO(k',V')
$$
(respectively,
$$
GS(k,V)\overset{t}\hookrightarrow G(k,V)
\overset{\phi'}\hookrightarrow G(l,U)\overset{\tau}\hookrightarrow GS(l,\tilde{U})\overset{\phi''}\hookrightarrow
GS(k',V')),
$$
where $t$ is the tautological embedding, $\phi'$ and $\phi''$ are standard extensions and $\tau$ is an
isotropic extension.
\end{definition}

Note that a combination of isotropic and standard extensions is always given by one of the formulas
$V_k\mapsto V_k\oplus W$ or
$V_k\mapsto V_k^\bot\oplus W$ for an appropriately chosen orthogonal (respectively, symplectic) isomorphism
$V'\simeq V\oplus\hat{W}$ and an isotropic subspace $W\subset\hat{W}$. Here $\bot$ refers to the orthogonal
(respectively, symplectic) structure on $V$. Furthermore, one easily proves the following lemma.

\begin{lemma}\label{compos of combin}
A composition of combinations of isotropic and standard extensions is a combination of isotropic and standard
extensions.
\end{lemma}

\begin{remark}\label{prop of st ext}
Let $\phi':X\to X'$ be a standard extension, where $X$ and $X'$ are both Grassmannians or, respectively, isotropic
Grassmannians of the same type. It is easy to see that, if $X$ and $X'$ are not (isomorphic to) projective spaces,
then $\phi'$ does not factor through an embedding of a projective space into $X'$. If $X$ and $X'$ are isotropic, then
$\phi'$ is not a combination of isotropic and standard extensions.
\end{remark}

{\bf Theorem 1}.
\textit{Let $X\simeq G(k,V)$, $X'\simeq G(k',V')$, or $X=GO(k,V),\ X'=GO(k',V')$, or
$X=GS(k,V),\ X'=GS(k',V')$, and let $\phi: X \to X'$ be a linear morphism. If $X=GO(k,V),\ X'=GO(k',V')$, assume
in addition that either $k\le[\frac{\dim V}{2}]-3$ and $k'\le[\frac{\dim V'}{2}]-3$, or that
$[\frac{\dim V'}{2}]-k'\le[\frac{\dim V}{2}]-k\le2$ and both $\dim V$ and $\dim V'$ are odd.
Then some of the following statements holds:\\
(i) $\phi$ is a standard extension;\\
(ii) $X$ and $X'$ are isotropic Grassmannians and $\phi$ is a combination of isotropic and standard
extensions;\\
(iii) $\phi$ factors through a projective space on $X'$ or, in case $X'=GO(k',V')$, through a maximal quadric
$Q_{\beta}^{\dim V'-2k'}$.}

\begin{proof}
We first consider in detail the case of symplectic Grassmannians.
The proof goes by induction on $k$. For $k=1$ the symplectic Grassmannian
$GS(1,V)$ equals $\mathbb{P}(V)$, hence the linear morphism $\phi$ maps it isomorphically onto a projective space in
$X'$. Therefore statement (iii) holds trivially in this case.

Assume now that $k\ge2$ and the assertion holds for $k-1$ and any $k'\ge1$.
Set $n:=\frac{1}{2}\dim V$, $n':=\frac{1}{2}\dim V'$, $Y_{\beta}:=GS(k-1,V)$.
Let $Z:=\{(V_{k-1},x)\in Y_{\beta}\times X|\ x\in\mathbb{P}(V_{k-1}^\bot/V_{k-1})\}\overset{p}\to Y_{\beta}$
be the family of projective spaces $\mathbb{P}^q_\beta,\ q=2n-2k+2,$ on $X\simeq GS(k,V)$.
Since $\phi$ is a linear morphism, for any $V_{k-1}\in Y_{\beta}$ the morphism
$\phi(\mathbb{P}(V_{k-1}^\bot/V_{k-1}))$ is a projective space on $X'$.
Therefore we obtain a family
$\tilde{Z}:=\{(V_{k-1},x)\in Y_{\beta}\times X'|\ x\in\phi(\mathbb{P}(V_{k-1}^\bot/V_{k-1}))\}
\overset{\tilde{p}}\to Y_{\beta}$ of $q$-dimensional projective spaces on $X'$.
We claim that

(a) all spaces of the family $\tilde{p}:\tilde{Z}\to Y_{\beta}$ lie in the spaces of the
family with base $Y'_{\alpha}:=GS(k'+1,V')$ (this is possible only if $k\le n-1$),\\
or

(b) all spaces of the family
$\tilde{p}:\tilde{Z}\to Y_{\beta}$ lie in the spaces of the family with base $Y'_{\beta}:=GS(k'-1,V')$.

Indeed, consider the varieties $\Sigma_{\alpha}:=\{(V_{k-1},V'_{k'+1})\in Y_{\beta}\times Y'_{\alpha}|\
\phi(\mathbb{P}(V_{k-1}^\bot/V_{k-1}))\subset\mathbb{P}((V'_{k'+1})^*)\}$
and $\Sigma_{\beta}:=\{(V_{k-1},V'_{k'-1})\in Y_{\beta}\times Y'_{\beta}|\ \phi(\mathbb{P}(V_{k-1}^\bot/V_{k-1})
\subset \mathbb{P}(V_{k'-1}^\bot/V_{k'-1})\}$
with natural projections
$Y_{\beta}\overset{p_\alpha}\leftarrow\Sigma_{\alpha}\overset{q_{\alpha}}\to Y'_{\alpha}$
and
$Y_{\beta}\overset{p_\beta}\leftarrow\Sigma_{\beta}\overset{q_{\beta}}\to Y'_{\beta}$.
By construction,
$\Sigma_{\alpha}$ is a closed subset of $Y_{\beta}\times Y'_{\alpha}$ and $p_{\alpha}$ is a projective morphism. Hence,
$W_{\alpha}:=p_{\alpha}(\Sigma_{\alpha})$ is a closed subset of $Y_{\beta}$. By a similar reason,
$W_{\beta}:=p_{\beta}(\Sigma_{\beta})$ is a closed subset of $Y_{\beta}$. Since any space of the family
$\tilde{Z}\to Y_{\beta}$ lies
in at least one maximal space on $X'$, it follows that $W_{\alpha}\cup W_{\beta}=Y_{\beta}$. However, $Y_{\beta}$ is
irreducible, therefore either $W_{\alpha}=Y_{\beta}$ (i.e. case (a) holds), or
$W_{\beta}=Y_{\beta}$ (i.e. case (b) holds).

We now consider the cases (a) and (b) separately.

In the case (a), by Lemma \ref{intersect dim=1},(i), each space of the family $\tilde{p}:\tilde{Z}\to Y_{\beta}$ lies
in a unique space $\mathbb{P}((V'_{k'+1})^*)$ of the family with base $Y'_{\alpha}$. This
means that $p_{\alpha}:\Sigma_{\alpha}\to Y_{\beta}$ is a bijective morphism, hence an isomorphism as $Y_{\beta}$ is a
smooth variety. Therefore, there is a well defined morphism
\begin{equation}\label{mor i}
\phi_{\alpha}:=q_{\alpha}\circ p_{\alpha}^{-1}:\ Y_{\beta}=GS(k-1,V)\to Y'_{\alpha}=GS(k'+1,V').
\end{equation}
Moreover, there is a commutative diagram
\begin{equation}\label{phi Y^2}
\xymatrix{
& \Gamma\ar[rrr]^{\phi_{\Gamma}}\ar[dl]_{p_1}\ar[dr]^{p_2} &&& \Gamma'_{\alpha}\ar[dl]_{\bar{p}_1}\ar[dr]^{\bar{p}_2}& \\
Y_{\beta}\ar[dr]&& X\ar[dr]& Y'_{\alpha} && X',\\
&\ar[r]_{\phi_{\alpha}}&\ar[ur]&\ar[r]_{\phi}&\ar[ur]& }
\end{equation}
where $\Gamma$ is the variety of isotropic $(k-1,k)$-flags in $V$,
$\Gamma'_{\alpha}$ is the variety of isotropic $(k',k'+1)$-flags in $V'$,
and $\phi_\Gamma,\ p_1,\ p_2,\ \bar{p}_1$ and $\bar{p}_2$ are the induced projections.

Assume that $\phi_{\alpha}$ is not a constant map. We first show that $\phi_{\alpha}$ is linear.
Fix $V_{k+1}\in GS(k+1,V)$ and a subspace $V_{k-2}$ of $V_{k+1}$.
Consider the
projective plane $\mathbb{P}^2_X:=\mathbb{P}((V_{k+1}/V_{k-2})^*)$ on $X$.
The points on $\mathbb{P}^2_X$ are $k$-dimensional subspaces
$U_k\subset V$ such that $V_{k-2}\subset U_k\subset V_{k+1}$.
According to Lemma \ref{Bkn sympl},(i), any $U_k$ defines a projective space $\mathbb{P}(U_k^*)$ on $Y_{\beta}$,
and also a projective line
$\mathbb{P}((U_k/V_{k-2})^*)$ on $Y_{\beta}$.
Fix $U_k$ and denote the projective line $\mathbb{P}((U_k/V_{k-2})^*)$ by $\mathbb{P}^1_{Y_{\beta}}$.
Furthermore, fix a projective line $\mathbb{P}^1_X$ in $\mathbb{P}^2_X$
and consider the rational curve $\mathbb{P}^1_{\Gamma}:=
\{(V_{k-1},V_k)\in\Gamma|\ V_{k-1}\in\mathbb{P}^1_{Y_{\beta}},\ V_k\in \mathbb{P}^1_X\}$ on $\Gamma$.
Diagram (\ref{phi Y^2}) yields a commutative diagram
\begin{equation}\label{curves}
\xymatrix{
\mathbb{P}^1_{Y_{\beta}}\ar[d]_-{\phi_{\alpha}|_{\mathbb{P}^1_{Y_{\beta}}}} &&
\mathbb{P}^1_{\Gamma}\ar[rr]^-{p_2|_{\mathbb{P}^1_{\Gamma}}}
\ar[d]_{\phi_{\Gamma}|_{\mathbb{P}^1_{\Gamma}}}\ar[ll]_-{p_1|_{\mathbb{P}^1_{\Gamma}}} &&
\mathbb{P}^1_X\ar[d]_{\phi|_{\mathbb{P}^1_X}}  \\
\phi_{\alpha}(\mathbb{P}^1_{Y_{\beta}}) &&
\phi_{\Gamma}(\mathbb{P}^1_{\Gamma})\ar[rr]^-{\bar{p}_2|_{\phi_{\Gamma}(\mathbb{P}^1_{\Gamma})}}
\ar[ll]_-{\bar{p}_1|_{\phi_{\Gamma}(\mathbb{P}^1_{\Gamma})}} && \phi(\mathbb{P}^1_X).}
\end{equation}
Since $\phi:X\to X'$ is linear, $\phi|_{\mathbb{P}^1_X}:\mathbb{P}^1_X\to \phi(\mathbb{P}^1_X)$ is an isomorphism.
Furthermore, $p_2|_{\mathbb{P}^1_{\Gamma}}$ is an isomorphism by construction. Therefore,
$\bar{p}_2|_{\phi_{\Gamma}(\mathbb{P}^1_{\Gamma})}$ and $\phi_{\Gamma}|_{\mathbb{P}^1_{\Gamma}}$ are isomorphisms.

We claim now that $p_1|_{\mathbb{P}^1_{\Gamma}}$ and $\bar{p}_1|_{\phi_{\Gamma}(\mathbb{P}^1_{\Gamma})}$ are also
isomorphisms. Indeed, for $p_1|_{\mathbb{P}^1_{\Gamma}}$ this holds by construction. Consider
$\bar{p}_1|_{\phi_{\Gamma}(\mathbb{P}^1_{\Gamma})}$. As $\phi(\mathbb{P}^1_X)$ is a projective line in $X'$, the
subspaces of $V'$ corresponding to the points of $\phi(\mathbb{P}^1_X)$ lie in some $k'+1$-dimensional subspace
$V'_{k'+1}$ of $V'$. This implies in view of Lemma \ref{intersect dim=1},(i) that, for any two distinct points
$V'_{k'},V''_{k'}\in\phi(\mathbb{P}^1_X)$, the projective spaces $\mathbb{P}({V'}_{k'}^\bot/V'_{k'})$ and
$\mathbb{P}({V''_{k'}}^\bot/V''_{k'})$ on $Y'_{\alpha}$ have $V'_{k'+1}$ as unique common point.
Note that, for each $V'_{k'}\in\phi(\mathbb{P}^1_X)$, $\mathbb{P}({V'_{k'}}^\bot/V'_{k'})$
is the isomorphic image under $\bar{p}_1$ of the projective space $\bar{p}_2^{-1}(V'_{k'})$, and that
$\bar{p}_2^{-1}(V'_{k'})\cap\phi_{\Gamma}(\mathbb{P}^1_{\Gamma})$ is a single point.
Hence, either $\bar{p}_1(\phi_{\Gamma}(\mathbb{P}^1_{\Gamma}))=\phi_{\alpha}(\mathbb{P}^1_{Y_{\beta}})$ equals the point
$V'_{k'+1}$, or $\bar{p}_1|_{\phi_{\Gamma}(\mathbb{P}^1_{\Gamma})}$ is an isomorphism. However, the former case is
impossible since $\phi_{\alpha}|_{\mathbb{P}^1_{Y_{\beta}}}$ is a non-constant, hence finite morphism by Lemma
\ref{finite}. Thus $\bar{p}_1|_{\phi_{\Gamma}(\mathbb{P}^1_{\Gamma})}$ is an isomorphism.

Diagram (\ref{curves}) implies now that $\phi_{\alpha}|_{\mathbb{P}^1_{Y_{\beta}}}$ is also an isomorphism.
To show that $\phi_{\alpha}$ is linear it suffices to prove that $\phi_{\alpha}(\mathbb{P}^1_{Y_{\beta}})$
is a projective line on $Y'_{\alpha}$. This latter fact follows directly from Lemma \ref{family CS},(ii)
applied to the following data: $\mathbb{P}^1=\mathbb{P}^1_{Y_{\beta}}$,
$x=\phi(\Span(\underset{V_{k-1}\in\mathbb{P}^1}\cup V_{k-1}))$, $C=\phi_{\alpha}(\mathbb{P}^1_{Y_{\beta}})$.

Note next that the diagram (\ref{phi Y^2}) allows to reconstruct $\phi$ from $\phi_{\alpha}$. Indeed, for any
$V_k\in X$ the projective space $\mathbb{P}(V_k^*)=p_1(p_2^{-1}(V_k))$ is mapped via
$\phi_{\alpha}$ to the unique projective space $\mathbb{P}({V'_{k'}}^\bot/V'_{k'})$ which contains
$\phi_{\alpha}(\mathbb{P}^{k-1}_{\alpha})$. The original morphism $\phi$ is precisely the map assigning $V'_{k'}$ to
$V_k\in X$.

We are now ready to apply the induction assumption to $\phi_{\alpha}$. Since $\phi_{\alpha}$ is linear we conclude that
there are the following three possibilities: (a.1) $\phi_{\alpha}$ is a standard extension, or (a.2) $\phi_{\alpha}$
factors through an isotropic extension, or (a.3) $\phi_{\alpha}$ factors through a morphism to a projective space in
$Y_\alpha$.

(a.1) In this case we have a fixed isomorphism $V'\simeq V\oplus \hat{W}$ and $\phi_{\alpha}$ is given by the formula
\begin{equation}\label{st eq 2}
V_{k-1}\mapsto V_{k-1}\oplus W
\end{equation}
for an isotropic subspace $W$ of $\hat{W}$ (see Remark \ref{intrinsic def}). Therefore, for any $V_k\in X$, the
space $\mathbb{P}^{k-1}_{\alpha}=\mathbb{P}(V_k^*)$ on $Y_{\beta}$ is embedded by $\phi_{\alpha}$ in the projective
space $\mathbb{P}^{k'+1}_{\alpha}=\mathbb{P}((V_k\oplus W)^*)$ on $Y'_{\alpha}$.
Since $\mathbb{P}(V_k^*)=p_1(p_2^{-1}(V_k))$, diagram (\ref{phi Y^2}) implies that
$\phi_{\alpha}(\mathbb{P}(V_k^*))\subset\mathbb{P}^{\dim V'-2k'-1}_{\beta}=\bar{p}_1(\bar{p}_2^{-1}(V'_{k'}))$, where
$V'_{k'}:=\phi(V_k)$. We have thus shown that $\phi_{\alpha}(\mathbb{P}(V_k^*))$ lies in the intersection of maximal
projective spaces from the distinct families $PS_{\alpha}(k'+1,V')$ and $PS_{\beta}(k'+1,V')$ on $Y'_{\alpha}$. Hence,
by Lemma \ref{intersect dim=1},(ii), $k-1=\dim\phi_{\alpha}(\mathbb{P}(V_k^*))\le1$, i.e. $k=2$. Therefore,
$X\simeq GS(2,V),\ Y_{\beta}=\mathbb{P}(V)$, and $\phi_{\alpha}$ is an embedding of $\mathbb{P}(V)$ into
$Y'_{\alpha}$. Then $\phi_{\alpha}(\mathbb{P}(V))$ lies in a unique maximal projective space
$\mathbb{P}((V'_{k'+2})^*)$ for an isotropic subspace $V'_{k'+2}$ of $V'$.
This yields a monomorphism $j:V\hookrightarrow(V'_{k'+2})^*$.
Now the above reconstruction of $\phi$ via $\phi_{\alpha}$ shows that $\phi$ decomposes as
\begin{equation}\label{decomp1}
X=GS(2,V)\overset{t}\hookrightarrow G(2,V)\overset{\tilde{j}}\hookrightarrow
G(2,(V'_{k'+2})^*)\simeq G(k',V'_{k'+2})\overset{\tau}\hookrightarrow GS(k',V')=X',
\end{equation}
where $t$ is the tautological
embedding, the embedding $\tilde{j}$ is induced by the monomorphism $j$, and the embedding $\tau$ is induced by the
embedding of $V'_{k'+2}$ in $V'$.
Hence, $\phi$ is a combination of isotropic and standard extensions. One checks that as a consequence $\phi_{\alpha}$
is also a combination of isotropic and standard extensions. However, this contradicts to Remark \ref{prop of st ext},
and we conclude that case (a.1) is impossible.

(a.2) In this case  $\phi_{\alpha}$ is given by one of the formulas
\begin{equation}\label{st eq 2a}
V_{k-1}\mapsto V_{k-1}\oplus W
\end{equation}
or
\begin{equation}\label{st eq 2b}
V_{k-1}\mapsto V^\bot_{k-1}\oplus W,
\end{equation}
where $\bot$ refers to the symplectic structure on $V$. If $\phi_{\alpha}$ is given by (\ref{st eq 2a}), then
for an arbitrary $V_k\in X$,
$\phi_{\alpha}(\mathbb{P}(V^*_k))\subset\mathbb{P}((V_k\oplus W)^*)$.
Assume that $\dim\phi_{\alpha}(\mathbb{P}(V^*_k))>1$. Then by Lemma \ref{intersect dim=1},(ii),
$\phi_{\alpha}(\mathbb{P}(V^*_k))\not\subset\mathbb{P}(\phi(V_k)^{\bot'}/\phi(V_k))$,
where the symbol $\bot'$ refers to the symplectic structure on $V'$.
On the other hand, in view of diagram
(\ref{phi Y^2}), $\phi_{\alpha}(\mathbb{P}(V^*_k))\subset\mathbb{P}(\phi(V_k)^{\bot'}/\phi(V_k))$.
This implies $k=2$. Therefore, $\phi_{\alpha}$ is a combination of isotropic and standard extensions of the form
$$
Y_{\beta}=\mathbb{P}(V)\overset{\phi'_{\alpha}}\hookrightarrow G(l,U)
\overset{\tau_{\alpha}}\hookrightarrow
GS(l,\tilde{U})\overset{\phi''_{\alpha}}\hookrightarrow GS(k'+1,V')=Y'_{\alpha}
$$
(see Definition \ref{isotr st ext}). Then using diagram (\ref{phi Y^2}) it is easy to check that $\phi$ is given by
the formula
$$
V_k\mapsto V_k\oplus W
$$
and is a combination of isotropic and standard extensions of the form
$$
X=GS(2,V)\overset{t}\hookrightarrow G(2,V)\overset{\phi'}\hookrightarrow G(l+1,U)
\overset{\tau}\hookrightarrow
GS(l+1,\tilde{U})\overset{\phi''}\hookrightarrow GS(k',V')=X'.
$$

If $\phi_{\alpha}$ is given by (\ref{st eq 2b}), then
for an arbitrary $V_k\in X$,
$\phi_{\alpha}(\mathbb{P}(V^*_k))=\mathbb{P}((V^\bot_k\oplus W)^{\bot'}/(V^\bot_k\oplus W))$.
In view of the diagram (\ref{phi Y^2}) $\phi$ is given in this case by the formula
$$
V_k\mapsto V^\bot_k\oplus W
$$
and is a combination of isotropic and standard extensions of the form
$$
X=GS(k,V)\overset{t}\hookrightarrow G(k,V)\overset{\bot}{\simeq}G(2n-k,V)\overset{\phi'}\hookrightarrow
G(l+1,U)\overset{\tau}\hookrightarrow GS(l+1,\tilde{U})\overset{\phi''}\hookrightarrow GS(k',V')=X'.
$$
In this way, (ii) holds under the assumption (a.2).

(a.3) Here $\phi_{\alpha}$ factors through a morphism to some projective space $\mathbb{P}^s$ in $Y'_{\alpha}$,
and we may assume without loss of generality that $\mathbb{P}^s$ is maximal. If $\mathbb{P}^s=
\mathbb{P}({V'}_{k'}^\bot/V'_{k'})$ for some $V'_{k'}\in X'$, then in view of diagram (\ref{phi Y^2}) $\phi$ is the
constant map $V_k\mapsto V'_{k'}$, contrary to the linearity of $\phi$. Hence, $\mathbb{P}^s=
\mathbb{P}((V'_{k'+2})^*)$ for some isotropic subspace $V'_{k'+2}$ of $V'$. On the other hand,
diagram(\ref{phi Y^2}) implies that the projective space
$\mathbb{P}(V_k^*)=p_1(p_2^{-1}(V_k))$ is mapped via $\phi_{\alpha}$ to the projective space
$\mathbb{P}({V'_{k'}}^\bot/V'_{k'})=\bar{p}_1({\bar{p}_2}^{-1}(V'_{k'}))$ for $V'_{k'}=\phi(V_k)$. Thus,
$\phi_{\alpha}(\mathbb{P}(V_k^*))\subset\mathbb{P}((V'_{k'+2})^*)\cap\mathbb{P}({V'_{k'}}^\bot/V'_{k'})$. By Lemma
\ref{intersect dim=1},(ii) this implies $k=2$.  Hence, $X=GS(2,V),\ Y_{\beta}=\mathbb{P}(V)$ and, since $\phi_{\alpha}$
is linear, it is an embedding $\mathbb{P}(V)\hookrightarrow \mathbb{P}((V'_{k'+2})^*)$ corresponding to
a monomorphism $j:V\hookrightarrow(V'_{k'+2})^*$. The above mentioned reconstruction of $\phi$ from $\phi_{\alpha}$
shows now that $\phi$ decomposes as
$$
X=GS(2,V)\overset{t}\hookrightarrow G(2,V)\overset{\tilde{j}}\hookrightarrow
G(2,(V'_{k'+2})^*)\simeq G(k',V'_{k'+2})\overset{\tau}\hookrightarrow GS(k',V')=X',
$$
where $t$ is the tautological
embedding, $\tilde{j}$ is the standard extension corresponding to the monomorphism $j$, and $\tau$ is an isotropic
extension corresponding to an embedding of an isotropic subspace $V'_{k'+2}$ in $V'$. This means that $\phi$ is
a combination of isotropic and standard extensions, i.e. statement (ii) holds.

To complete case (a) it remains to consider the possibility that $\phi_{\alpha}$ is a constant map,
i.e. $\phi_{\alpha}(Y_{\beta})=\{V'_{k'+1}\}$ for some $V'_{k'+1}\subset V'$. Then diagram (\ref{phi Y^2})
implies that $\phi(X)$ lies in the projective space $\mathbb{P}(({V'}_{k'+1})^*)$ on $X'$, i.e. (iii) holds.

We now proceed to the case (b). In this case,
by Lemma \ref{intersect dim=1},(i) each space of the family $\tilde{p}:\tilde{Z}\to Y_{\beta}$ lies in a unique
projective space $\mathbb{P}(V_{k'-1}^\bot/V_{k'-1})$ of the family with base $Y'_{\beta}$. This
means that $p_{\beta}:\Sigma_{\beta}\to Y_{\beta}$ is a bijective morphism, hence an isomorphism as $Y_{\beta}$ is
a smooth variety. Therefore, there is a well-defined morphism
\begin{equation}\label{1mor i}
\phi_{\beta}:=q_{\beta}\circ p_{\beta}^{-1}:\ Y_{\beta}=GS(k-1,V)\to Y'_{\beta}=GS(k'-1,V'),
\end{equation}
and a commutative diagram similar to (\ref{phi Y^2})
\begin{equation}\label{phi Y^1}
\xymatrix{
& \Gamma\ar[rrr]^{\phi_{\Gamma}}\ar[dl]_{p_1}\ar[dr]^{p_2}&&& \Gamma'_{\beta}\ar[dl]_{p'_1}\ar[dr]^{p'_2}&
\\
Y_{\beta}\ar[dr]&& X\ar[dr]& Y'_{\beta} && X',\\
&\ar[r]_{\phi_{\beta}}&\ar[ur]&\ar[r]_{\phi}&\ar[ur]& }
\end{equation}
where $\phi,\ \Gamma,\ p_1,\ p_2,$ are as in (\ref{phi Y^2}),
$\Gamma'_{\beta}$ is the variety of isotropic $(k'-1,k')$-flags in $V'$,
and $\phi_\Gamma,\ p'_1$ and $p'_2$ are the induced projections.

Assume that $\phi_{\beta}$ is a non-constant morphism.
Then $\phi_{\beta}$ is linear, and the proof is similar to that of the linearity of $\phi_{\alpha}$.
Indeed, consider the diagram analogous to (\ref{curves}) with $\phi_{\alpha},\bar{p}_1,\bar{p}_2$ replaced respectively
by $\phi_{\beta},p'_1,p'_2$. By essentially the same argument as above, this is a commutative diagram of isomorphisms.
The fact that $\phi_{\beta}(\mathbb{P}^1_{Y_{\beta}})$ is a projective line on $Y'_{\beta}$ follows from Lemma
\ref{family CS},(i) for the data $\mathbb{P}^1=\phi(\mathbb{P}^1_X)$,
$x=\phi(\Span(\underset{V_{k-1}\in\mathbb{P}^1}\cup V_{k-1}))$, $C=\phi_{\beta}(\mathbb{P}^1_{Y_{\beta}})$.

The morphism $\phi_{\beta}$ maps a projective space $\mathbb{P}(V_k^*)$ to a unique projective space,
and thus reconstructs $\phi$ in an obvious way.

Now, by the induction assumption, (b.1) $\phi_{\beta}$ is a standard extension, or (b.2) $\phi_{\beta}$ is a
combination of isotropic and standard extensions, or (b.3) $\phi_{\beta}$ factors through a linear morphism into some
projective space $\mathbb{P}^s$ in $Y'_{\beta}$. Consider these three cases (b.1)-(b.3).

(b.1) In this case $\phi_{\beta}$ is a standard extension.
Using the reconstruction of $\phi$ via $\phi_{\beta}$ mentioned above, one immediately sees that $\phi$ is also a
standard extension.

(b.2) In this case $\phi_{\beta}$ is a combination of isotropic and standard extensions, and, using the reconstruction
of $\phi$ via $\phi_{\beta}$, the reader will check that $\phi$ also is a combination of isotropic and standard
extensions.

(b.3) In this case $\phi_{\beta}$ factors through a linear morphism of $Y_{\beta}$ into some maximal projective space
$\mathbb{P}^s$ on $Y'_\beta$.
Then $\mathbb{P}^s=\mathbb{P}^s_\beta:=\mathbb{P}({V'}_{k'-2}^\bot/V'_{k'-2})$ for some $V'_{k'-2}\subset V'$, or
$\mathbb{P}^s=\mathbb{P}^s_{\alpha}:=G(k'-1,V'_{k'})$ for some isotropic subspace $V'_{k'}\subset V'$.
The second case is clearly impossible because it would imply that
$\phi$ maps $X$ into the single point $V'_{k'}$, contrary to linearity of $\phi$. Hence,
$\mathbb{P}^s=\mathbb{P}^s_{\beta}$.

Fix $V_k\in X$ and set $V'_{k'}:=\phi(V_k).$ Diagram (\ref{phi Y^1}) shows that the projective
space $\mathbb{P}(V_k^*)=p_1(p_2^{-1}(V_k))$ is embedded by $\phi_{\beta}$ into the
intersection of the maximal projective spaces $\mathbb{P}^s_{\beta}$ and
$\mathbb{P}^{k'-1}_{\alpha}:=\mathbb{P}((V'_{k'-1})^*)=p'_1({p'}_2^{-1}(V'_{k'}))$ in $Y'_{\beta}$.
By Lemma \ref{intersect dim=1},(ii) this implies $k=2$,
i.e. $X=GS(2,V),\ Y_{\beta}=\mathbb{P}(V)$, and
$\phi_{\beta}:\ \mathbb{P}(V)\to\mathbb{P}^s_{\beta}=\mathbb{P}({V'}_{k'-2}^\bot/V'_{k'-2})$ is a linear embedding
induced by a certain monomorphism $f:V\to{V'}_{k'-2}^\bot/V'_{k'-2}$.
Diagram (\ref{phi Y^1}) shows now that $\phi$ is the composition
$$
X=GS(2,V)\overset{i}\to GS(2,{V'}_{k'-2}^\bot/V'_{k'-2})\overset{\tilde{\phi}}\hookrightarrow GS(k',V')=X',
$$
where $i$ is induced by $f$ and $\tilde{\phi}$ is the standard extension corresponding to the flag
${V'}_{k'-2}\subset{V'}_{k'-2}^\bot$ in $V'$. Being a composition of standard extensions, $\phi$ is itself a standard
extension, i.e. (i) holds.

To complete the proof in the symplectic case it remains to consider the possibility the $\phi_{\beta}$
is a constant morphism. Let $\phi_{\beta}(Y_{\beta})=\{V'_{k'-1}\}$ for some $V'_{k'-1}\subset V'$.
Then $\phi(X)$ lies in the projective space $\mathbb{P}({V'}_{k'-1}^\bot/V'_{k'-1})$ on $X'$, i.e. (iii) holds.

\vspace{2mm}
We now briefly outline the changes needed in the proof for the orthogonal case. The main idea is to replace the family
of projective spaces $PS_{\beta}(k,V)$ by the family of maximal quadrics $QO_{\beta}(k,V)$ on $X$. Note first that the
image of a quadric $Q^{\dim V-2k}_{\beta}$ under a linear morphism is either a quadric or a projective space.
Using this and the additional conditions imposed on $k,k',\dim V,\dim V'$
we show that $\phi$ induces a well defined linear morphism of the form
\begin{equation}\label{phi a}
\phi_{\alpha}: QO_{\beta}(k,V)=GO(k-1,V)\to GO(k'+1,V')
\end{equation}
or
\begin{equation}\label{phi b}
\phi_{\beta}: QO_{\beta}(k,V)=GO(k-1,V)\to GO(k'-1,V').
\end{equation}
The above conditions ensure that $\phi$ does not map maximal quadrics of the form $Q^{\dim V-2k}_{\beta}$ into maximal
quadrics of the form $Q^4_{\gamma}$.

The linearity of $\phi_{\alpha}$ and $\phi_{\beta}$, provided that they are
non-constant morphisms, is proved by arguments similar to the above using Lemma \ref{family CO} instead of Lemma
\ref{family CS}.
The rest of the proof goes along the same lines as in the symplectic case. When working with maximal quadrics
$Q^{\dim V-2k}_{\beta}$ on $GO(k,V)$ instead of maximal projective spaces $\mathbb{P}^{2n-2k+1}_{\beta}$ on $GS(k,V)$,
one uses Lemmas \ref{Bkn},(iv) and \ref{two possib},(ii) instead of Lemmas \ref{Bkn sympl},(ii) and
\ref{intersect dim=1},(ii).

\vspace{2mm}
Finally, we leave the case $X\simeq G(k,V)$ and $X'\simeq G(k',V')$ entirely to the reader.
\end{proof}

\begin{corollary}\label{new 3.8}
\textit{Let $X\simeq G(k,V)$, $X'\simeq G(k',V')$, or $X=GO(k,V),\ X'=GO(k',V')$, or
$X=GS(k,V),\ X'=GS(k',V')$, and let $\phi: X \to X'$ be a linear morphism.
If $X=GO(k,V),\ X'=GO(k',V')$, assume in addition that either $k\le[\frac{\dim V}{2}]-3$ and
$k'\le[\frac{\dim V'}{2}]-3$, or that $[\frac{\dim V'}{2}]-k'\le[\frac{\dim V}{2}]-k\le2$ and both $\dim V$ and
$\dim V'$ are odd. Then $\phi$ is an embedding unless it factors through a projective space on $X'$
or through a maximal quadric when $X'=GO(k',V')$.}
\end{corollary}

\begin{corollary}\label{3.8}
\textit{Let $X\simeq G(k,V)$, $X'\simeq G(k',V')$, or $X=GO(k,V),\ X'=GO(k',h')$, or
$X=GS(k,V),\ X'=GS(k',V')$, and let $\phi: X \to X'$ be a linear embedding.
If $X=GO(k,V),\ X'=GO(k',V')$, assume in addition that either $k\le[\frac{\dim V}{2}]-3$ and
$k'\le[\frac{\dim V'}{2}]-3$, or that $[\frac{\dim V'}{2}]-k'\le[\frac{\dim V}{2}]-k\le2$ and both $\dim V$ and
$\dim V'$ are odd. Then some of the following statements holds:\\
(i) $\phi$ is a standard extension;\\
(ii) $X$ and $X'$ are isotropic Grassmannians and $\phi$ is a combination of isotropic and standard
extensions;\\
(iii) $\phi$ factors through a projective space on $X'$ or, in case $X'=GO(k',V')$, through a maximal quadric
$Q_{\beta}^{\dim V'-2k'}$.}
\end{corollary}

\begin{remark}\label{X,X'}
Note that if $X\simeq G(k,V)$, $X'\simeq G(k',V')$ and $\phi:X\to X'$ is an embedding, the
statement of
Corollary \ref{3.8} simplifies as follows: $\phi$ is either a standard extension, or factors through a projective
space on $X'$ (cf. Proposition 3.1 in \cite{PT1}).
\end{remark}

\begin{remark}\label{X'}
If $X=GS(k,V)$, $X'=GS(\frac{\dim V'}{2},V')$ and $\phi:X\to X'$ is a linear morphism, then $k=\frac{\dim V}{2}$.
This follows easily from Lemmas \ref{Bkn sympl},(iii) and \ref{finite}.
\end{remark}

We will also need the following partial extension of Theorem 1.

\begin{proposition}\label{case n-2 etc}
Let $\dim V=2n\ge10$, $\dim V'=2n'$ and $\phi:X=GO(n-2,V)\to X'=GO(n'-2,V')$ be a linear embedding.
Then some of the following statements holds:\\
(i) $\phi$ is a standard extension;\\
(ii) $X$ and $X'$ are isotropic Grassmannians and $\phi$ is a combination of isotropic and standard
extensions;\\
(iii) $\phi$ factors through a projective space on $X'$, through a maximal quadric $Q_{\beta}^{\dim V'-2k'}$, or
through the Grassmannian $G(n'-2,V'_{n'})\subset X'$ for a maximal isotropic subspace $V'_{n'}$ of $V'$.
\end{proposition}
\begin{proof}
Considering the image of the family $QO_{\beta}(n-2,V)$ under $\phi$, we see similarly to the proof of Theorem 1,
that at least one of the following morphisms
\begin{equation}\label{phi alpha}
\phi_{\alpha}: QO_{\beta}(n-2,V)=GO(n-3,V)\to PO_{\alpha}(n'-2,V'),
\end{equation}
\begin{equation}\label{phi beta}
\phi_{\beta}: QO_{\beta}(n-2,V)=GO(n-3,V)\to GO(n'-3,V'),
\end{equation}
\begin{equation}\label{phi gamma}
\phi_{\gamma}: QO_{\beta}(n-2,V)=GO(n-3,V)\to QO_{\gamma}(n'-2,V')
\end{equation}
must be well defined.

Assume that $\phi_{\alpha}$ is well defined. Then one sees that an obvious analog of diagram (\ref{phi Y^2})
applies also in the case we consider here. Set $V'_{n'-2}:=\phi(V_{n-2})$ for $V_{n-2}\in X$.
Note that $p_2^{-1}(V_{n-2})=\mathbb{P}(V_{n-2}^*)$is mapped under $\phi_{\Gamma}$ into
${p'}_2^{-1}(V'_{n'-2})\simeq\mathbb{P}^1\times\mathbb{P}^1$. Since $n\ge5$,
this map is a constant map. Hence $\phi_{\alpha}$ maps the projective space $\mathbb{P}(V_{n-2}^*)$ into a point.
Lemma \ref{finite} implies now that $\phi_{\alpha}$ is a constant map.
i.e. $\phi_{\alpha}(QO_{\beta}(n-2,V))=\{V'_{n'-1}\}$ for some $V'_{n'-1}\subset V'$. Then the analog of diagram
(\ref{phi Y^2}) implies that $\phi(X)$ lies in the projective space $\mathbb{P}((V'_{n'-1})^*)$ on $X'$, i.e.
statement (iii) holds.

Next, if $\phi_{\beta}$ is well defined, then one applies Theorem 1 to $\phi_{\beta}$ and recovers $\phi$ from
$\phi_{\beta}$ as in the proof of Theorem 1.

In the remainder of the proof we assume that $\phi_{\gamma}$ is well defined. We start by constructing a diagram
analogous to (\ref{phi Y^2}):
\begin{equation}\label{phi Y^3}
\xymatrix{
& \bar{\Gamma}\ar[rrr]^{\phi_{\bar{\Gamma}}}\ar[dl]_{\pi_1}\ar[dr]^{\pi_2} &&&
\bar{\Gamma}'\ar[dl]_{\pi'_1}\ar[dr]^{\pi'_2}& \\
Y\ar[dr]&& X\ar[dr]& Y' && X'.\\
&\ar[r]_{\phi_Y}&\ar[ur]&\ar[r]_{\phi}&\ar[ur]& }
\end{equation}
By definition, $\bar{\Gamma}$ is a fixed connected component of the variety of isotropic $(n-2,n)$-flags in $V$, and
$\bar{\Gamma}'$ is a fixed connected component of the variety of isotropic $(n'-2,n')$-flags in $V$. Next, we define
$Y$. For this we fix codimension 1 subspace $\tilde{V}$ in $V$ such that the symmetric form $\Phi|_{\tilde{V}}$ is
non-degenerate, and set $Y:=GO(n-1,\tilde{V})$. Similarly we define $Y'$ as $GO(n'-1,\tilde{V}')$. The projections
$\pi_1$, $\pi_2$, $\pi'_1$, $\pi'_2$ are as follows:
$\pi_1:(V_{n-2}\subset V_n)\mapsto V_n\cap\tilde{V}$, $\pi_2:(V_{n-2}\subset V_n)\mapsto V_{n-2}$,
$\pi'_1:(V'_{n'-2}\subset V'_{n'})\mapsto V'_{n'}\cap\tilde{V}'$, $\pi'_2:(V'_{n'-2}\subset V'_{n'})\mapsto V'_{n'-2}$.
To define the morphisms  $\phi_Y$ and $\phi_{\bar{\Gamma}}$, consider
a point $V_n\cap\tilde{V}\in Y$. By construction, the fibre $\pi_1^{-1}(V_n\cap\tilde{V})$ is isomorphic to the
Grassmannian $G(n-2,V_n)$ which is isomorphically mapped onto $\pi_2(G(n-2,V_n))$. The composition
$G(n-2,V_n)\overset{\pi_2}\to \pi_2(G(n-2,V_n))\overset{\phi}\hookrightarrow X'
\overset{t} \hookrightarrow G(n'-2,V')$, where $t$ is the tautological embedding, is a linear embedding of
Grassmannians, hence by Theorem 1 it is either a standard extension or factors through an embedding into a projective
space. In both cases one sees that there is a unique isotropic subspace $V'_{n'}$ of $V'$ such that
$(\phi\circ\pi_2)(G(n-2,V_n))\subset G(n'-2,V'_{n'})$. Define now $\phi_Y:Y\to Y'$ by setting
$\phi_Y(V_n\cap\tilde{V})=V'_{n'}\cap\tilde{V}'$. The morphism $\phi_{\bar{\Gamma}}:\bar{\Gamma}\to\bar{\Gamma}'$ is
then recovered by the commutativity of diagram (\ref{phi Y^3}).

Assume now that the morphism $\phi_Y$ is finite. Consider a point $V_{n-2}\in X$ and set $V'_{n'-2}=\phi(V_{n-2})$.
By diagram (\ref{phi Y^3}) the projective line $\mathbb{P}^1:=\pi_1(\pi_2^{-1}(V_{n-2}))$ on $Y$ is mapped into the
projective line $\mathbb{P}'^1:=\pi'_1({\pi'}_2^{-1}(V'_{n'-2}))$ on $Y'$. Since the morphism $\phi_Y|_{\mathbb{P}^1}$
is finite, it follows that this morphism is surjective. This implies that the morphism
$\phi_{\bar{\Gamma}}:\bar{\Gamma}\to\bar{\Gamma}'$ maps fibres of  $\pi_2$ onto fibres of $\pi'_2$.

Next, fix a point $V_{n-3}\in GO(n-3,V)$. The maximal quadric $GO(1,V_{n-3}^\bot/V_{n-3})$ is mapped by $\phi$ onto the
quadric $Q^4_{\gamma}$ corresponding to the isotropic flag $\phi_{\gamma}(V_{n-3})$. Consequently, according to
the above stated property of $\phi_{\bar{\Gamma}}$ the variety $\pi_2^{-1}(GO(1,V_{n-3}^\bot/V_{n-3}))$ is mapped by
$\phi_{\bar{\Gamma}}$ onto the variety ${\pi'}_2^{-1}(Q^4_{\gamma})$. Hence
$\pi_1(\pi_2^{-1}(GO(1,V_{n-3}^\bot/V_{n-3})))$ is mapped by $\phi_Y$ onto $\pi'_1({\pi'}_2^{-1}(Q^4_{\gamma}))$.
However, one can check that the variety $\pi_1(\pi_2^{-1}(GO(1,V_{n-3}^\bot/V_{n-3})))$ is isomorphic to $\mathbb{P}^3$,
while the variety $\pi'_1({\pi'}_2^{-1}(Q^4_{\gamma}))$ is 5-dimensional. This is a contradiction.

Hence $\phi_Y$ is not finite, and Lemma \ref{finite} implies that $\phi_Y$ is a constant map. Set $V'_{n'}=\phi_Y(Y)$.
Then diagram (\ref{phi Y^3}) yields that $\phi(X)\subset \pi'_2({\pi'}_1^{-1}(V'_{n'-2}))=G(n'-2,V'_{n'})$, and
statement (iii) holds.
\end{proof}

\vspace{1cm}

\section{Linear ind-Grassmannians}\label{stand Gr}

\vspace{1cm}

\noindent

Recall that an \textit{ind-variety} is the direct limit
$\mathbf{X}=\underset{\longrightarrow}\lim X_m$
of a chain of morphisms of algebraic varieties
\begin{equation}\label{eq1}
X_1\stackrel{\phi_1}{\to}X_2\stackrel{\phi_2}{\to}\dots\stackrel{\phi_{m-1}}{\to}X_m\stackrel{\phi_m}{\to}
X_{m+1}\stackrel{\phi_{m+1}}{\to}\dots\ .
\end{equation}
Note that the direct limit of the chain (\ref{eq1})
does not change if we replace the sequence
$\{X_m\}_{m\ge1}$ by a subsequence $\{X_{i_m}\}_{m\ge1}$ and the morphisms $\phi_m$ by the compositions
$\tilde{\phi}_{i_m}:={\phi}_{i_{m+1}-1}\circ...\circ{\phi}_{i_m+1}\circ{\phi}_{i_m}$.
Let $\mathbf{X}$ be the direct limit of (\ref{eq1}) and  $\mathbf{X}'$ be the direct limit of a chain
\begin{equation}\label{eq1'}
X'_1\stackrel{\phi'_1}{\to}X'_2\stackrel{\phi'_2}{\to}\dots\stackrel{\phi'_{m-1}}{\to}X'_m\stackrel{\phi'_m}{\to}
X'_{m+1}\stackrel{\phi'_{m+1}}{\to}\dots\ .
\end{equation}
A {\it morphism of ind-varieties} $\mathbf{f}:\mathbf{X}\to\mathbf{X}'$ is a map from
$\underset{\to}\lim X_n$ to $\underset{\to}\lim X'_n$ induced by a collection of morphisms of algebraic varieties
$\{f_m:X_m\to Y_{n_m}\}_{m\ge1}$
such that $\psi_{n_m}\circ f_m=f_{m+1}\circ\phi_m$ for all $m\ge1$.
The identity morphism $\mathrm{id}_\mathbf{X}$ is a morphism which induces the identity as a set-theoretic map from
$\mathbf{X}$ to $\mathbf{X}'$.
A morphism $\mathbf{f}:\mathbf{X}\to\mathbf{X}'$ is an \textit{isomorphism} if there exists a morphism
$\mathbf{g}:\mathbf{X}'\to\mathbf{X}$ such that $\mathbf{g}\circ\mathbf{f}=\mathrm{id}_\mathbf{X}$ and
$\mathbf{f}\circ\mathbf{g}=\mathrm{id}_{\mathbf{X}'}$.

In what follows we only consider chains (\ref{eq1}) such that $X_m$ are complete algebraic varieties,
$\underset{n\to\infty}\lim(\dim X_n)=\infty$, and the morphisms $\phi_m$ are embeddings. We call such ind-varieties
\textit{locally complete}. Furthermore, we call a morphism
$\mathbf{f}:\mathbf{X}=\underset{\to}\lim X_n\to\mathbf{X}'=\underset{\to}\lim X'_n$ of locally complete
ind-varieties an \textit{embedding} if all morphisms $f_m:X_m\to X'_{n_m},\ m\ge1,$ are embeddings.
\begin{definition}\label{lin ind-Gr}
A \textit{linear ind-Grassmannian} is an ind-variety $\mathbf{X}$ obtained as a direct limit of a chain of embeddings
$$
X_1\overset{\phi_1}{\hookrightarrow}X_2\overset{\phi_2}{\hookrightarrow}\dots\overset{\phi_{m-1}}{\hookrightarrow}
X_m\overset{\phi_m}{\hookrightarrow}X_{m+1}\overset{\phi_{m+1}}{\hookrightarrow}\dots
$$
where each $X_m$ is a Grassmannian or an isotropic Grassmannian, $\underset{n\to\infty}\lim(\dim X_n)=\infty$,
and all embeddings $\phi_m$ are linear morphisms.
\end{definition}
Note that Definition \ref{lin ind-Gr} allows for a "mixture" of all three types of Grassmannians (usual Grassmannians,
orthogonal Grassmannians, symplectic Grassmannians). Note also that when considering orthogonal Grassmannians we
restrict ourselves to connected orthogonal Grassmannians with Picard group isomorphic to $\mathbb{Z}$, see
\ref{linear ortho ind-Grassm}.

We now define certain standard Grassmannians and isotropic Grassmannians.

\begin{definition}\label{G(k)}
Fix an infinite chain of vector spaces
$$
V_{n_1}\subset V_{n_2}\subset...\subset V_{n_m}\subset V_{n_{m+1}}\subset...
$$
of dimensions $n_m,\ n_m< n_{m+1}$.

a) For an integer $k$, $1\le k< n_1$, set $\mathbf{G}(k):=\underset{\to}\lim G(k,V_{n_m})$ where
$$
G(k,V_{n_1})\hookrightarrow G(k,V_{n_2})\hookrightarrow...\hookrightarrow G(k,V_{n_m})\hookrightarrow
G(k,V_{n_{m+1}})\hookrightarrow...
$$
is the chain of canonical inclusions of Grassmannians.

b) For a sequence of integers $1\le k_1< k_2<...$ such that $ k_m< n_m,$  $\underset{m\to\infty}\lim(n_m-k_m)=\infty$,
set
$\mathbf{G}(\infty):=\underset{\to}\lim G(k_m,V_{n_m})$ where
$$
G(k_1,V_{n_1})\hookrightarrow G(k_2,V_{n_2})\hookrightarrow...\hookrightarrow G(k_m,V_{n_m})\hookrightarrow
G(k_{m+1},V_{n_{m+1}})\hookrightarrow...
$$
is an arbitrary chain of standard extensions of Grassmannians.

c) Assume that $V_{n_m}$ are endowed with compatible non-degenerate symmetric (respectively, symplectic) forms $\Phi_m$.
In the symplectic case $\frac{1}{2}n_m\in\mathbb{Z}_+$.
For an integer $k,\ 1\le k\le[\frac{n_1}{2}]$, set
 $\mathbf{G}\mathrm{O}(k,\infty):=\underset{\to}\lim GO(k,V_{n_m})$ (respectively,
 $\mathbf{G}\mathrm{S}(k,\infty):=\underset{\to}\lim GS(k,V_{n_m})$) where
$$
GO(k,V_{n_1})\hookrightarrow GO(k,V_{n_2})\hookrightarrow...\hookrightarrow GO(k,V_{n_m})\hookrightarrow
GO(k,V_{n_{m+1}})\hookrightarrow...
$$
(respectively,
$$
GS(k,V_{n_1})\hookrightarrow GS(k,V_{n_2})\hookrightarrow...\hookrightarrow GS(k,V_{n_m})\hookrightarrow
GS(k,V_{n_{m+1}})\hookrightarrow...)
$$
is the chain of canonical inclusions of isotropic Grassmannians.

d) For a sequence of integers $1\le k_1< k_2<...$ such that $k_m<[\frac{n_m}{2}],$
$\underset{m\to\infty}\lim([\frac{n_m}{2}]-k_m)=\infty$,
set $\mathbf{G}\mathrm{O}(\infty,\infty)=\underset{\to}\lim G\mathrm{O}(k_m,V_{n_m})$
(respectively, $\mathbf{G}\mathrm{S}(\infty,\infty):=\underset{\to}\lim GS(k_m,V_{n_m})$)
where
\begin{equation}\label{chain O}
GO(k_1,V_{n_1})\hookrightarrow GO(k_2,V_{n_2})\hookrightarrow...\hookrightarrow GO(k_m,V_{n_m})\hookrightarrow
GO(k_{m+1},V_{n_{m+1}})\hookrightarrow...
\end{equation}
(respectively,
\begin{equation}\label{chain S}
GS(k_1,V_{n_1})\hookrightarrow GS(k_2,V_{n_2})\hookrightarrow...\hookrightarrow GS(k_m,V_{n_m})\hookrightarrow
GS(k_{m+1},V_{n_{m+1}})\hookrightarrow...)
\end{equation}
is an arbitrary chain of standard extensions of isotropic Grassmannians.

e) In the symplectic case, consider a sequence of integers $1\le k_1< k_2<...$ such that $k_m<\frac{n_m}{2}$,
$\underset{m\to\infty}\lim(\frac{n_m}{2}-k_m)=k\in\mathbb{N}$, and set
$\mathbf{G}\mathrm{S}(\infty,k):=\underset{\to}\lim GS(k_m,V_{n_m})$
for any chain of standard extensions (\ref{chain S}). In the orthogonal case, assume first that $\dim V_{n_m}$ are even.
Then set $\mathbf{G}\mathrm{O}^0(\infty,k):=\underset{\to}\lim GO(k_m,V_{n_m})$ for a chain (\ref{chain O}) where
$k_m<\frac{n_m}{2}$, $\underset{m\to\infty}\lim(\frac{n_m}{2}-k_m)=k\in\mathbb{N},\ k\ge2$. Finally, consider the
orthogonal case under the assumption that $\dim V_{n_m}$ are odd. Then set
$\mathbf{G}\mathrm{O}^1(\infty,k):=\underset{\to}\lim GO(k_m,V_{n_m})$ for a chain (\ref{chain O}) where
$k_m<[\frac{n_m}{2}]$, $\underset{m\to\infty}\lim([\frac{n_m}{2}]-k_m)=k\in\mathbb{N}$.
\end{definition}

The \textit{infinite projective space} $\mathbf{P}^\infty$ is defined as the ind-variety $\mathbf{G}(1)$.
Note that $\mathbf{P}^\infty\simeq\mathbf{G}\mathrm{S}(1)$. When writing $\mathbf{G}\mathrm{O}^0(\infty,k)$ below
we automatically assume $k\ne1$.

\begin{lemma}\label{G infty}
All standard ind-Grassmannians $\mathbf{G}(\infty)$, $\mathbf{G}\mathrm{O}(\infty,\infty)$,
$\mathbf{G}\mathrm{S}(\infty,\infty)$, $\mathbf{G}(k)$, $\mathbf{G}\mathrm{O}(k,\infty)$,
$\mathbf{G}\mathrm{S}(k,\infty)$, $\mathbf{G}\mathrm{O}^0(\infty,k)$, $\mathbf{G}\mathrm{O}^1(\infty,k)$,
$\mathbf{G}\mathrm{S}(\infty,k)$,
are well defined. In other words, a standard Grassmannian does not depend,
up to an isomorphism of ind-varieties, on the specific chain of standard embeddings used in its definition.
\end{lemma}

\begin{proof}
We consider only $\mathbf{G}(\infty)$. All other cases are similar.
Let two chains of strict standard extensions
$$
G(k_1,V_{n_1})\overset{\phi_1}\hookrightarrow G(k_2,V_{n_2})\overset{\phi_2}\hookrightarrow...
\overset{\phi_{m-1}}\hookrightarrow G(k_m,V_{n_m})\overset{\phi_m}\hookrightarrow G(k_{m+1},V_{n_{m+1}}
)\overset{\phi_{m+1}}\hookrightarrow...,
$$
$$
G(k'_1,V_{n'_1})\overset{\phi'_1}\hookrightarrow G(k'_2,V_{n'_2})\overset{\phi'_2}\hookrightarrow...
\overset{\phi'_{m-1}}\hookrightarrow G(k'_m,V_{n'_m})\overset{\phi'_m}\hookrightarrow G(k'_{m+1},V_{n'_{m+1}}
)\overset{\phi'_{m+1}}\hookrightarrow...,
$$
such that
$$
\underset{m\to\infty}\lim k_m=\underset{m\to\infty}\lim k'_m=\underset{m\to\infty}\lim(n_m-k_m)=
\underset{m\to\infty}\lim(n'_m-k'_m)=\infty,
$$
be given. We will show that their respective direct limits $\mathbf{G}(\infty)$ and $\mathbf{G}'(\infty)$ are
isomorphic as ind-varieties.

For this, we have to construct two infinite subsequences
$\{i_s\}_{s\ge1}$ and $\{j_s\}_{s\ge1}$ of $\mathbb{Z}_+$
and two sets of morphisms
$\mathbf{f}=\{f_s:G(k_{i_s},V_{n_{i_s}})\to G(k'_{j_s},V'_{n'_{j_s}})\}_{s\ge1}$,
$\mathbf{g}=\{g_m:G(k'_{j_s},V'_{n'_{j_s}})\to G(k_{i_{s+1}},V_{n_{i_{s+1}}})\}_{m\ge1}$
such that they determine morphisms of ind-varieties $\mathbf{f}:\mathbf{G}(\infty)\to\mathbf{G}'(\infty),
\mathbf{g}:\mathbf{G}'(\infty)\to\mathbf{G}(\infty)$ with
$\mathbf{g}\circ\mathbf{f}=\mathrm{id}_{\mathbf{G}(\infty)}$ and
$\mathbf{f}\circ\mathbf{g}=\mathrm{id}_{\mathbf{G}'(\infty)}$.
Assume that the desired subsequences $\{i_s\}_{s\ge1},\{j_s\}_{s\ge1}$ and
morphisms $f_l,g_l$ are constructed for $1\le l\le s-1$,
and that these morphisms are strict standard extensions. Denote for short
$k:=k_{i_s},\ n:=n_{i_s},\ V:=V_n,\ k':=k'_{j_s},\ n':=n'_{j_s},\ V':=V'_{n'},\
G:=G(k,V),\ G':=G(k',V'),\ f:=f_s:G\hookrightarrow G',\ \tilde{k}:=k_{i_{s+1}},
\tilde{n}:=n_{i_{s+1}},\ \tilde{V}:=V_{\tilde{n}},\ \tilde{G}:=G(\tilde{k},\tilde{V}),\
\phi:=\phi_{i_s}:G\hookrightarrow\tilde{G}$. Without loss of generality that we assume that $\tilde{k}>k'$.
By Remark \ref{intrinsic def}, $f$ is given by a triple $(W_f,U_f,\underline{f})$, where $W_f\subset U_f$
is a flag in $V'$. Respectively, $\phi$ is given by a triple $(W_\phi,U_\phi,\underline{\phi})$,
where $W_\phi\subset U_\phi$ is a flag in $\tilde{V}$.

For the induction step we will now find a strict standard extension
$g:=g_s:G'\hookrightarrow\tilde{G}$ such that $g\circ f=\phi$. Indeed,
consider the exact triples
$0\to W_f\to U_f\xrightarrow{\underline{f}}V\to0,\ \ \
0\to W_{\phi}\to U_{\phi}\xrightarrow{\underline{\phi}}V\to0$.
Since both $\underline{f}$ and $\underline{\phi}$ are epimorphisms, and $\dim U_{\phi}>\dim U_f$ as $\tilde{k}>k'$,
it follows that there exists a (non-unique) epimorphism $\varepsilon_U:U_{\phi}\twoheadrightarrow U_f$
such that $\underline{\phi}=\underline{f}\circ\varepsilon_U$.
Then $\varepsilon_U|_W$ is a well-defined epimorphism $W_{\phi}\twoheadrightarrow W_f$.
Putting $W_g:=\ker{\varepsilon_U}$, we have the exact triple
$0\to W_g\to U_{\phi}\xrightarrow{\varepsilon_U}U_f\to0$.
Next, set
$U'_g:=W_g\oplus V'$ and fix an embedding $i:U'_g\hookrightarrow\tilde{V}$ such that $i|_{U_f}=\textrm{id}$.
Then $W_g\subset U_g:=i(U'_g)$ is a flag in $\tilde{V}$ equipped with an isomorphism $\underline{g}:U_g/W_g\simeq V'$.
The corresponding strict standard extension $g:G'\hookrightarrow\tilde{G}$ satisfies the property $g\circ f=\phi$,
as claimed.
\end{proof}

Note furthermore that the standard ind-Grassmannians introduced above are isomorphic to certain ind-varieties
introduced in \cite{DiP}. More precisely, let $\widetilde{V}$ be a countable-dimensional vector space with basis
$\{v_1,...,v_n,...\}$ and let
$\widetilde{W}\subset\widetilde{V}$ be a subspace generated by a subset of $\{v_1,...,v_n,...\}$. Then
$G(\widetilde{W},\widetilde{V})$ is by definition the set of subspaces $\widetilde{E}\subset\widetilde{V}$
satisfying the following two conditions:

(i) $\Span(\{v_1,...,v_n,...\}\cap\widetilde{E})$ is of finite codimension in $\widetilde{E}$;

(ii) there exists a finite-dimensional subspace $\widetilde{U}\subset\widetilde{V}$ such that
$\widetilde{W}\subset\widetilde{E}+\widetilde{U}$, $\widetilde{E}\subset\widetilde{W}+\widetilde{U}$,
$\dim(\widetilde{E}\cap\widetilde{U})=\dim(\widetilde{W}\cap\widetilde{U})$.\\
Then it is easy to see (a much stronger result is proved in \cite{DiP}) that $G(\widetilde{W},\widetilde{V})$
has a natural structure of an ind-variety such that $G(\widetilde{W},\widetilde{V})$ is the direct limit of a chain of
standard extensions of Grassmannians. Moreover,
$$
G(\widetilde{W},\widetilde{V})\cong\mathbf{G}(\min\{\dim\widetilde{W},\codim_{\widetilde{V}}\widetilde{W}\}).
$$
Similarly, in the isotropic case (i.e. in the case when $\widetilde{W}$ is equipped with an appropriate
non-degenerate quadratic form)
the standard isotropic ind-Grassmannians introduced in this paper represent all isomorphism
classes of ind-varieties $G(\widetilde{W},\widetilde{V})$  introduced in \cite{DiP}
(in this case $\widetilde{W}$ is an isotropic subspace of $\widetilde{V}$) and satisfying
$\Pic G(\widetilde{W},\widetilde{V})\simeq\mathbb{Z}$.

\vspace{1cm}

\section{Classification of linear ind-Grassmannians}\label{sec5}

\vspace{0.5cm}

In this section we prove the following main result of the note.

\textbf{Theorem 2}. \textit{Every linear ind-Grassmannian is isomorphic as an ind-variety to one of the standard
ind-Grassmannians $\mathbf{G}(k)$ for $k\ge1$, $\mathbf{G}(\infty)$, $\mathbf{G}\mathrm{O}(k,\infty)$ for $k\ge1$,
$\mathbf{G}\mathrm{O}^0(\infty,k)$ for $k\ge2$, $\mathbf{G}\mathrm{O}^1(\infty,k)$ for $k\ge0$,
$\mathbf{G}\mathrm{O}(\infty,\infty)$, $\mathbf{G}\mathrm{S}(k,\infty)$ for $k\ge2$, $\mathbf{G}\mathrm{S}(\infty,k)$
for $k\ge0$, $\mathbf{G}\mathrm{S}(\infty,\infty)$, and the latter are pairwise non-isomorphic}.

\begin{proof}
Let a linear ind-Grassmannian $\mathbf{X}$ be given as the direct limit of a chain of embeddings
$$
X_1\stackrel{\phi_1}{\hookrightarrow}X_2\stackrel{\phi_2}{\hookrightarrow}
\dots\stackrel{\phi_{m-1}}{\hookrightarrow}X_m\stackrel{\phi_m}{\hookrightarrow}
X_{m+1}\stackrel{\phi_{m+1}}{\hookrightarrow}\dots,
$$
where $X_m$ are Grassmannians, possibly orthogonal or symplectic, such that
$\underset{m\to\infty}\lim(\dim X_m)=\infty.$
Then, for infinitely many $m$, $X_m$ will be a Grassmannian, or an orthogonal Grassmannian, or a symplectic
Grassmannian. Therefore, without loss of generality, we can assume that all $X_m$ are of one of the above three types.

Suppose first that all $X_m$ are Grassmannians. Then we have the following two options: for infinitely many $m$,
the embedding $\phi_m:X_m\to X_{m+1}$ factors through an embedding of a projective space into $X_{m+1}$, i.e.
there exists a commutative diagram of embeddings
$$
\xymatrix{
X_m \ar[rr]^-{\phi_m}\ar[dr]& & X_{m+1}\\
& \mathbb{P}^{j_m},\ar[ur] & }
$$
or this is not the case. In the first case $\mathbf{X}\simeq\underset{\to}\lim~\mathbb{P}^{j_m}$, hence
$\mathbf{X}\simeq\mathbf{P}^{\infty}$.
In the second case, by deleting some first embeddings we can assume that none of the embeddings $\phi_m:X_m\to X_{m+1}$
factors through an embedding of a projective space into $X_{m+1}$. Then, Corollary \ref{3.8} implies that all embeddings
$\phi_m$ are standard extensions, hence $\mathbf{X}$ is isomorphic to $\mathbf{G}(k)$ or $\mathbf{G}(\infty)$.

In the symplectic case, the reader will argue in a similar way that Corollary \ref{3.8} implies that $\mathbf{X}$ is
either isomorphic to $\mathbf{G}(k)$ or $\mathbf{G}(\infty)$ (this happens when all $\phi_m$ are combinations of
isotropic and standard extensions or factor through projective spaces), or to one of the standard symplectic
ind-Grassmannians.

The orthogonal case is similar but has some special features. First, if all morphisms $\phi_m$ factor through
respective quadrics $Q_{\beta}^{\dim V_{m+1}-2k_{m+1}}$, one needs to prove that the direct limit of any chain
of linear embeddings
$$
Q_1{\hookrightarrow}Q_2{\hookrightarrow}\dots{\hookrightarrow}Q_m{\hookrightarrow}Q_{m+1}{\hookrightarrow}\dots\ ,
$$
where $\underset{m\to\infty}\lim\dim Q_m=\infty$, is isomorphic either to $\mathbf{P}^{\infty}$ or to
$\mathbf{G}\mathrm{O}(1,\infty)$. This is an exercise which we leave to the reader. Second, in the orthogonal case one
applies Corollary \ref{3.8} when $[\frac{\dim V_m}{2}]-k_m\ge3$ for infinitely many $m$ (in this case one can assume
without loss of generality that $[\frac{\dim V_m}{2}]-k_m\ge3$ for all $m$). The case when
$[\frac{\dim V_m}{2}]-k_m\le2$ for infinitely many $m$ needs special attention. In the latter case one assumes without
loss of generality that $[\frac{\dim V_m}{2}]-k_m$ is constant and then applies Theorem 1 when $\dim V_m$ is odd for
all $m$, and Proposition \ref{case n-2 etc} when $\dim V_m$ is even for all $m$ (in the latter case
$\frac{\dim V_m}{2}-k_m=2$ for all $m$).

The first claim of Theorem 2 is now proved.

The claim that the standard ind-Grassmannians are pairwise non-isomorphic follows from Lemmas \ref{non-isom1},
\ref{non-isom2} and \ref{non-isom3} below.
\end{proof}

In what follows we will sometimes write $\mathbf{G}\mathrm{O}(\infty,k)$ meaning $\mathbf{G}\mathrm{O}^0(\infty,k)$
or $\mathbf{G}\mathrm{O}^1(\infty,k)$. This allows the simultaneous consideration of $\mathbf{G}\mathrm{O}^0(\infty,k)$
and $\mathbf{G}\mathrm{O}^1(\infty,k)$.

\begin{lemma}\label{non-isom1}
(i) Let $k,k'\in\mathbb{Z}_{+}\cup\{\infty\},\ k\ne k'$. Then
$\mathbf{G}(k)\not\simeq\mathbf{G}(k')$, $\mathbf{G}\mathrm{O}(k,\infty)\not\simeq\mathbf{G}\mathrm{O}(k',\infty)$,
$\mathbf{G}\mathrm{O}(k,\infty)\not\simeq\mathbf{G}\mathrm{O}(k',\infty)$.

(ii) Let $k\ge2$. Then $\mathbf{G}\mathrm{O}^0(\infty,k)\not\simeq\mathbf{G}\mathrm{O}^1(\infty,k)$.

(iii) Let $k,k'\in\mathbb{N}\cup\{\infty\},\ k\ne k'$. Then $\mathbf{G}\mathrm{O}(\infty,k)\not\simeq
\mathbf{G}\mathrm{O}(\infty,k')$, $\mathbf{G}\mathrm{S}(\infty,k)\not\simeq\mathbf{G}\mathrm{S}(\infty,k')$.

(iv) Let $k\in\mathbb{N}\cup\{\infty\},\ k'\in\mathbb{Z}_{+}\cup\{\infty\},\ k\ne k'$.
Then $\mathbf{G}\mathrm{O}(\infty,k)\not\simeq\mathbf{G}\mathrm{O}(k',\infty)$, $\mathbf{G}\mathrm{S}(\infty,k)
\not\simeq\mathbf{G}\mathrm{S}(k',\infty)$.

\end{lemma}
\begin{proof}
In (i), (iii) and (iv) we only consider the symplectic case and leave the other cases to the reader.

(i) Let $k>k'$. Assume that $k\in\mathbb{Z}_{+}$ and that $\mathbf{X}:=\mathbf{G}\mathrm{S}(k,\infty)$ and
$\mathbf{X}':=\mathbf{G}\mathrm{S}(k',\infty)$ are isomorphic. This implies that there exist
subsequences $\{i_s\}_{s\ge1}$ and $\{j_s\}_{s\ge1}$ of $\mathbb{Z}_{+}$ and
a chain of linear embeddings
\begin{equation}\label{chain1}
...\hookrightarrow GS(k,V_{n_{i_s}})\overset{f_s}\hookrightarrow GS(k',V'_{n'_{j_s}})\overset{g_s}\hookrightarrow
GS(k,V_{n_{i_{s+1}}})\overset{f_{s+1}}\hookrightarrow GS(k',V'_{n'_{j_{s+1}}})\overset{g_{s+1}}\hookrightarrow ...\ ,
\end{equation}
such that the compositions $g_s\circ f_s$ and $f_{s+1}\circ g_s$ are standard extensions and
the direct limit of the chain (\ref{chain1}) is isomorphic to both $\mathbf{X}$ and $\mathbf{X}'$.
According to Corollary \ref{3.8}, we can assume without loss of generality that all embeddings $f_s$ and $g_s$ are
standard extensions, or factor through isotropic extensions, or factor through embeddings to projective spaces.

In the first case, since
$GS(k,V_{n_{i_s}})\overset{f_s}\hookrightarrow GS(k',V'_{n'_{j_s}})$
is a standard extension, it follows from (\ref{stand ineq1}) that $k'\ge k$, contrary to the assumption.

In the third case both $\mathbf{X}$ and $\mathbf{X}'$ are isomorphic to $\mathbf{P}^\infty$.
On the other hand, Remark \ref{prop of st ext} implies that $\mathbf{X}$ is not isomorphic to
$\mathbf{P}^\infty$ as $k>1$.

Consider now the second case. Here $f_s$ factorizes as
$f_s: GS(k,V_{n_{i_s}})\overset{t}\hookrightarrow G(k,V_{n_{i_s}})\overset{\tilde{f}_s}\hookrightarrow
GS(k',V'_{n'_{j_s}})$,
where $t$ is the tautological embedding and $\tilde{f}_s$ is an isotropic extension followed by a standard extension.
The composition
\begin{equation}\label{GkV}
G(k,V_{n_{i_s}})\overset{\tilde{f}_s}\hookrightarrow GS(k',V'_{n'_{j_s}})\overset{\tilde{t}}\hookrightarrow
G(k',V'_{n'_{j_s}}),
\end{equation}
$\tilde{t}$ being the tautological embedding, is a standard extension or factors through a projective space.
The latter assumption leads to the same contradiction as in the above considered third case, so we must assume that
(\ref{GkV}) is a standard extension. The existence of a standard extension
$G(k,V_{n_{i_s}})\hookrightarrow G(k',V'_{n'_{j_s}})$ implies $k'\ge k,\ k'-k\le n'_{j_s}-n_{i_s},$ or
$k\le n'_{j_s}-k'$, $k+k'\ge n_{i_s}$ (see Remark \ref{prop of st ext1}). Since for $n_{i_s}$ large enough, both pairs
of inequalities contradict our assumption that $k>k'$, we conclude that the second case is also impossible.

We have now shown that all three cases lead to contradictions, hence (i) follows for $k\in\mathbb{Z}_{+}$.
The argument for $k=\infty$ is very similar.

(ii) The maximal quadrics on $\mathbf{G}\mathrm{O}^0(\infty,k)$ not lying in projective spaces have dimension $2k$,
while $\mathbf{G}\mathrm{O}^1(\infty,k)$ not lying in projective spaces have dimension $2k+1$, see Lemma \ref{Bkn}.
This imlies that $\mathbf{G}\mathrm{O}^0(\infty,k)\not\simeq\mathbf{G}\mathrm{O}^1(\infty,k)$.

(iii) Let $\mathbb{Z}_{+}\ni k>k'$. Assume that $\mathbf{X}:=\mathbf{G}\mathrm{S}(k,\infty)$ and
$\mathbf{X}':=\mathbf{G}\mathrm{S}(k',\infty)$ are
isomorphic. As above, this implies that there exists a chain of linear embeddings (\ref{chain1})
such that the compositions $g_s\circ f_s$ and $f_{s+1}\circ g_s$ are standard extensions and
the direct limit of the chain (\ref{chain1}) is isomorphic to both $\mathbf{X}$ and $\mathbf{X}'$.
Without loss of generality we can assume that all embeddings
$f_s$ and $g_s$ are standard extensions, or factor through isotropic extensions, or factor through embeddings
to projective spaces.

In the first case we have a standard extension
$$
GS(\frac{1}{2}\dim V_{n_{i_s}}-k,V_{n_{i_s}})\overset{f_s}\hookrightarrow GS(\frac{1}{2}\dim V'_{n'_{j_s}}-
k',V'_{n'_{j_s}}),
$$
and (\ref{stand ineq2}) gives $k\le k'$, contrary to the assumption.

The arguments in the second and third case are similar to the respective arguments in (i).

The proof is finished for $k<\infty$. The case $k=\infty$ is similar.

(iv) The argument is practically the same as in (i).

\end{proof}

\begin{lemma}\label{non-isom2}
For any $k,k'\in\mathbb{Z}_{+}\cup\{\infty\},k''\in\mathbb{N}\cup\{\infty\}$ the following assertions hold.

(i) $\mathbf{G}(k)\not\simeq\mathbf{G}\mathrm{S}(k',\infty)$, unless $k=k'=1$,\
$\mathbf{G}(k)\not\simeq\mathbf{G}\mathrm{O}(k',\infty)$,

(ii) $\mathbf{G}(k)\not\simeq\mathbf{G}\mathrm{O}(\infty,k'')$,\
$\mathbf{G}(k)\not\simeq\mathbf{G}\mathrm{S}(\infty,k'')$.

\end{lemma}
\begin{proof} Again we consider only the symplectic case and leave the orthogonal case to the reader.

(i) We have to prove that $\mathbf{G}(k)\not\simeq\mathbf{G}\mathrm{S}(k',\infty)$, unless $k=k'=1$.
The case $k'=1,\ k>k'$, is already considered in Lemma \ref{non-isom1},(i), so we can assume $k\ne1,\ k'\ne1,\ k\ne k'$.

Let $\mathbf{G}(k)$ (respectively, $\mathbf{G}\mathrm{S}(k',\infty)$) be given as the direct limit of a chain of strict
standard extensions
$$
G(k,V_{n_1})\hookrightarrow G(k,V_{n_2})\hookrightarrow...\hookrightarrow G(k,V_{n_m})\hookrightarrow
G(k,V_{n_{m+1}})\hookrightarrow...
$$
(respectively,
$$
GS(k',V'_{n'_1})\overset{\phi_1}\hookrightarrow GS(k',V'_{n'_2})\overset{\phi_2}\hookrightarrow...
\overset{\phi_{m-1}}\hookrightarrow GS(k',V'_{n'_m})\overset{\phi_m}\hookrightarrow
GS(k',V'_{n'_{m+1}})\overset{\phi_{m+1}}\hookrightarrow...).
$$
Suppose that $\mathbf{G}(k)\simeq\mathbf{G}\mathrm{S}(k',\infty)$. This means
that there exist two infinite subsequences
$\{i_s\}_{s\ge1}$ and $\{j_s\}_{s\ge1}$ of $\mathbb{Z}_+$
and two sets of morphisms
$\mathbf{f}=\{f_s:G(k,V_{n_{i_s}})\to GS(k',V'_{n'_{j_s}})\}_{s\ge1}$,
$\mathbf{g}=\{g_s:GS(k',V'_{n'_{j_s}})\to G(k,V_{n_{i_{s+1}}})\}_{m\ge1}$
which determine morphisms of ind-varieties $\mathbf{f}:\mathbf{G}(k)\to\mathbf{G}\mathrm{S}(k',\infty),
\mathbf{g}:\mathbf{G}\mathrm{S}(k',\infty)\to\mathbf{G}(k)$ with
$\mathbf{g}\circ\mathbf{f}=\mathrm{id}_{\mathbf{G}(k)}$ and
$\mathbf{f}\circ\mathbf{g}=\mathrm{id}_{\mathbf{G}\mathrm{S}(k',\infty)}$.

Set $\tilde{V}:=V_{n_{i_{s+1}}},\ \tilde{G}:=G(k,\tilde{V}),\ V':=V'_{n'_{j_s}},
\ GS:=GS(k',V'),\ \tilde{V}':=V'_{n'_{j_{s+1}}},\ \widetilde{GS}:=GS(k',\tilde{V}'),\ g:=g_s:GS\hookrightarrow\tilde{G},
\ f:=f_{s+1}:\tilde{G}\hookrightarrow\widetilde{GS},
\ \phi:=\phi_{i_s}:GS\hookrightarrow\widetilde{GS}$. Note that $\phi$ is a standard extension and $\phi=f\circ g$ by
construction.

Consider the composition
$F:\tilde{G}\overset{f}\hookrightarrow \widetilde{GS}\overset{i}\hookrightarrow
G(k',\tilde{V}')$ where $i$ is the tautological embedding.

The morphism $F$ is a linear embedding, hence, by Corollary \ref{3.8}, we may assume
without loss of generality that

(a) $F$ is a standard extension,\\
or

(b) $F$ factors through an embedding into a projective space.

Consider these two cases.

(a) By Remark \ref{intrinsic def2},
$\phi$ is given by a triple $(W_{\phi},U_{\phi},\underline{\phi})$ where $W_{\phi}\subset U_{\phi}$ is a flag in
$\tilde{V}'$. Furthermore, without loss of generality we may assume that $F$ is given by a triple
$(W_f,U_f,\underline{F})$ for a flag $W_f\subset U_f$ in $\tilde{V}'$. Since $\phi(GS)=f\circ g(GS)\subset
F(\tilde{G})$, the following chain of inclusions holds:
$$
W_F\subset W_{\phi}\subset U_{\phi}\subset U_F\subset \tilde{V}'.
$$
Therefore we have an embedding $U_{\phi}/W_{\phi}\hookrightarrow U_F/W_{\phi}$ and a projection
$U_F/W_F\twoheadrightarrow U_F/W_{\phi}$. However, since $\phi$ is a standard extension, the fixed symplectic form
$\tilde{\Phi}'$ on $\tilde{V}'$ induces a nondegenerate form on $U_{\phi}/W_{\phi}$, while it induces the zero form on
$U_F/W_F$ as $f(\tilde{G})\subset\widetilde{GS}$.
This contradiction shows that the case (a) is impossible.

(b) By assumption, $F:\tilde{G}\overset{f}\hookrightarrow\widetilde{GS}\overset{i}\hookrightarrow G(k',\tilde{V}')$
decomposes as
$\tilde{G}\hookrightarrow\mathbb{P}^r\hookrightarrow G(k',\tilde{V}')$. Without loss of generality we assume that
$\mathbb{P}^r$ is a maximal projective space on $G(k',\tilde{V}')$, and consider the two possible cases:
$\mathbb{P}^r=\{V_{k'}\subset\tilde{V}'|V_{k'-1}\subset V_{k'}\subset V^\bot_{k'-1}\}$ and
$\mathbb{P}^r=\{V_{k'}\subset\tilde{V}'$
$|V_{k'}\subset V_{k'+1}\}$
for some fixed subspaces $V_{k'-1}$ and $V_{k'+1}$ of $\tilde{V}'$, $V_{k'-1}$ being isotropic.

In the former case any
$V_{k'}\in G(k',\tilde{V}')$ such that $V_{k'-1}\subset V_{k'}\subset V_{k'-1}^\bot$
is isotropic, i.e. $V_{k'}\in\widetilde{GS}\cap\mathbb{P}^r$. In other words,
$$
\widetilde{GS}\cap\mathbb{P}^r=\mathbb{P}(V_{k'-1}^\bot/V_{k'-1}),
$$
where the intersection is taken in $G(k',\tilde{V}')$. This means that $\phi$ factors through
a projective subspace of $\widetilde{GS}$, which contradicts Remark \ref{prop of st ext}.
Hence, the former case is impossible.

In the latter case it is easy to check that, for $n'_{j_{s+1}}=\dim\tilde{V}'>2$, the subspace
$V_{k'+1}\subset \tilde{V}'$ is necessarily isotropic. Then $\widetilde{GS}\cap\mathbb{P}^r=\mathbb{P}((V_{k'+1})^*)$,
and we are led to a contradiction as in the former case.

(ii) The proof is analogous to the proof of (i) and we leave it to the reader.
\end{proof}

\begin{lemma}\label{point, line}
Let $1\le k<n=[\dim V/2]$ and $\phi:GO(k,V)\to GO(k',V'),\ V_k\mapsto V_k\oplus W$, be a standard extension.
Let two maximal projective spaces $\mathbb{P}^k_{\alpha}$ and $\mathbb{P}^{n-k}_{\beta}$ intersect in a point.
Then there exist maximal projective spaces $\mathbb{P}^{k'}_{\alpha}$ and
$\mathbb{P}^{n'-k'}_{\beta},\ n'=[\dim V'/2]$, on
$GO(k',V')$ such that $\phi(\mathbb{P}^k_{\alpha})\subset\mathbb{P}^{k'}_{\alpha}$,
$\phi(\mathbb{P}^{n-k}_{\beta})\subset\mathbb{P}^{n'-k'}_{\beta}$, and
$\mathbb{P}^{k'}_{\alpha}\cap\mathbb{P}^{n'-k'}_{\beta}$ is a point.
\end{lemma}

\begin{proof}
The projective spaces $\mathbb{P}^k_{\alpha}$ and $\mathbb{P}^{n-k}_{\beta}$ determine a configuration
$V_{k-1},V_{k+1},V_n$ as in Lemma \ref{two possib},(iv). The subspaces $V_{k-1}\oplus W,V_{k+1}\oplus W,V_n\oplus W$
of $V'$ form the configuration which determines the desired projective spaces $\mathbb{P}^{k'}_{\alpha}$ and
$\mathbb{P}^{n'-k'}_{\beta}$.
\end{proof}

\begin{lemma}\label{non-isom3}

(i) $\mathbf{G}\mathrm{O}(k,\infty)\not\simeq\mathbf{G}\mathrm{S}(k',\infty)$ for $k,k'\in\mathbb{Z}_{+}\cup\{\infty\}$.

(ii) $\mathbf{G}\mathrm{O}(k,\infty)\not\simeq\mathbf{G}\mathrm{S}(\infty,k')$ for $k\in\mathbb{Z}_{+}\cup\{\infty\},
k'\in\mathbb{N}\cup\{\infty\}$.

(iii) $\mathbf{G}\mathrm{O}(\infty,k)\not\simeq\mathbf{G}\mathrm{S}(k',\infty)$ for
$k\in\mathbb{N}\cup\{\infty\},k'\in\mathbb{Z}_{+}\cup\{\infty\}$.

(iv) $\mathbf{G}\mathrm{O}(\infty,k)\not\simeq\mathbf{G}\mathrm{S}(\infty,k')$ for $k,k'\in\mathbb{N}\cup\{\infty\}$.
\end{lemma}

\begin{proof} We consider in detail only the case of  $\mathbf{G}\mathrm{O}(\infty,\infty)$ and
$\mathbf{G}\mathrm{S}(\infty,\infty)$.
Let $\mathbb{P}^q$ for $q\ge2$ be a projective space on $\mathbf{G}\mathrm{O}(\infty,\infty)$ (respectively,
$\mathbf{G}\mathrm{S}(\infty,\infty)$). We now explain how to label $\mathbb{P}^q$ as $\mathbb{P}^q_{\alpha}$ or
$\mathbb{P}^q_{\beta}$. Fix an arbitrary chain of standard extensions
\begin{equation}\label{exh1}
GO(k_1,V_{n_1})\hookrightarrow GO(k_2,V_{n_2})\hookrightarrow...\hookrightarrow GO(k_m,V_{n_m})\hookrightarrow
GO(k_{m+1},V_{n_{m+1}})\hookrightarrow...
\end{equation}
(respectively,
\begin{equation}\label{exh2}
GS(k'_1,V_{n'_1})\hookrightarrow GS(k'_2,V_{n'_2})\hookrightarrow...\hookrightarrow GS(k'_m,V_{n'_m})\hookrightarrow
GS(k'_{m+1},V_{n'_{m+1}})\hookrightarrow...)
\end{equation}
such that
$$
\underset{m\to\infty}\lim k_m=\underset{m\to\infty}\lim(n_m-k_m)=\infty
$$
(respectively,
$$
\underset{m\to\infty}\lim k'_m=\underset{m\to\infty}\lim(n'_m-k'_m)=\infty)
$$
and
$\underset{\to}\lim GO(k_m,V_{n_m})=\mathbf{G}\textrm{O}(\infty,\infty)$ (respectively,
$\underset{\to}\lim GS(k'_m,V_{n'_m})=\mathbf{G}\textrm{S}(\infty,\infty)$). Without loss of generality we assume
that all $n_m$ in (\ref{exh1}) are odd.

Consider some $n_m$ such that $\mathbb{P}^q\subset GO(k_m,V_{n_m})$ (respectively, $\mathbb{P}^q\subset
GS(k'_m,V_{n'_m})$) and choose a maximal projective space $\mathbb{P}^r$ on $GO(k_m,V_{n_m})$ (respectively,
$GS(k'_m,V_{n'_m}$) such that $\mathbb{P}^q\subset\mathbb{P}^r$. The projective space $\mathbb{P}^r$ is either of type
$\mathbb{P}^r_{\alpha}$ or $\mathbb{P}^r_{\beta}$,  and we label $\mathbb{P}^q$ according to the label of
$\mathbb{P}^r$.
Lemma \ref{two possib},(i),(iii) (respectively, Lemma \ref{intersect dim=1},(i),(ii)) implies that this labeling is well
defined as long as the chain (\ref{exh1}) (respectively,  (\ref{exh2})) is fixed. Moreover, using Theorem 1 and
Lemma \ref{non-isom2} one can verify that the labelings $\mathbb{P}^q_{\alpha}$ and $\mathbb{P}^q_{\beta}$ are intrinsic
to the ind-variety $\mathbf{G}\mathrm{O}(\infty,\infty)$ (respectively, $\mathbf{G}\mathrm{S}(\infty,\infty)$), i.e. do
not depend on the
choice of chain (\ref{exh1}) (respectively,  (\ref{exh2})) satisfying the above conditions.

Let now $\mathbf{P}^{\infty}\hookrightarrow\mathbf{G}\mathrm{O}(\infty,\infty)$ (respectively,
$\mathbf{P}^{\infty}\hookrightarrow\mathbf{G}\mathrm{S}(\infty,\infty)$) be a linear embedding. We call its image an
\textit{infinite projective space $\mathbf{P}^{\infty}$ on} $\mathbf{G}\mathrm{O}(\infty,\infty)$ (respectively,
$\mathbf{G}\mathrm{S}(\infty,\infty)$). We say that $\mathbf{P}^{\infty}=\mathbf{P}^{\infty}_{\alpha}$ if
$\mathbf{P}^{\infty}=\underset{\to}\lim\mathbb{P}^q_{\alpha}$ for some projective spaces $\mathbb{P}^q_{\alpha}$ on
$\mathbf{G}\mathrm{O}(\infty,\infty)$ (respectively, $\mathbf{G}\mathrm{S}(\infty,\infty)$). In a similar way we define
$\mathbf{P}^{\infty}_{\beta}$ on $\mathbf{G}\mathrm{O}(\infty,\infty)$ (respectively,
$\mathbf{G}\mathrm{S}(\infty,\infty)$).

Next, we observe that Lemma \ref{point, line} implies that on $\mathbf{G}\mathrm{O}(\infty,\infty)$ there are pairs of
maximal infinite projective spaces $\mathbf{P}^{\infty}_{\alpha}$ and $\mathbf{P}^{\infty}_{\beta}$ such that
$\mathbf{P}^{\infty}_{\alpha}\cap\mathbf{P}^{\infty}_{\beta}$ is a point.

To complete the proof, we observe that on $\mathbf{G}\mathrm{S}(\infty,\infty)$ any two maximal infinite projective
spaces $\mathbf{P}^{\infty}_{\alpha}$ and $\mathbf{P}^{\infty}_{\beta}$ intersect in a projective line whenever their
intersection is non-empty. This follows from Lemma \ref{intersect dim=1}. More precisely, an infinite projective space
$\mathbf{P}^{\infty}_{\alpha}$ (respectively, $\mathbf{P}^{\infty}_{\beta}$) is maximal on $\mathbf{G}S(\infty,\infty)$
if and only if, for any chain (\ref{exh2}) the intersections
$\mathbf{P}^{\infty}_{\alpha}\cap GS(k'_m,V_{n'_mj})$ are maximal projective spaces in $GS(k'_m,V_{n'_m})$ for large
enough $m$. This is a consequence of Lemma \ref{two possib},(i). Now Lemma \ref{two possib},(iii) implies the assertion
that maximal projective spaces $\mathbf{P}^{\infty}_{\alpha}$ and $\mathbf{P}^{\infty}_{\beta}$ intersect in a
projective line whenever their intersection is non-empty.

Since the intersection properties of maximal infinite projective spaces $\mathbf{P}^{\infty}_{\alpha}$ and
$\mathbf{P}^{\infty}_{\beta}$ on $\mathbf{G}\mathrm{O}(\infty,\infty)$ and $\mathbf{G}\mathrm{S}(\infty,\infty)$
are intrinsic to the geometry of $\mathbf{G}\mathrm{O}(\infty,\infty)$ and $\mathbf{G}\mathrm{S}(\infty,\infty)$, we
conclude that $\mathbf{G}\mathrm{O}(\infty,\infty)$ and $\mathbf{G}\mathrm{S}(\infty,\infty)$ are non-isomorphic
ind-varieties.

The arguments in all other cases are similar. One either shows that on one of the ind-varieties in question there are
maximal projective spaces which do not exist on the other, or shows that the intersection properties of maximal
projective spaces are different on both ind-varieties. For instance, on $\mathbf{G}\mathrm{O}(k,\infty)$ there are
maximal projective spaces $\mathbb{P}^k_{\alpha}$ and $\mathbf{P}^{\infty}_{\beta}$ which intersect in a point, while on
$\mathbf{G}\mathrm{S}(k,\infty)$ two maximal projective spaces  $\mathbb{P}^k_{\alpha}$ and
$\mathbf{P}^{\infty}_{\beta}$ intersect in a projective line or do not intersect at all. We leave the details to the
reader.

\end{proof}

\end{document}